\documentclass[12pt]{amsart}
\usepackage{amsmath}
\usepackage{commath}
\usepackage[english]{babel}
\usepackage{amssymb}

\usepackage[dvips]{graphicx}
 \usepackage{enumerate}
\usepackage{mathrsfs}
\usepackage{units}
\usepackage{stmaryrd}

\usepackage[colorlinks=true, pdfstartview=FitV, linkcolor=blue,
citecolor=blue, urlcolor=blue]{hyperref}

\usepackage[utf8x]{inputenc}

\usepackage{tgtermes}
\usepackage[T1]{fontenc}

\calclayout

\theoremstyle{plain}
\newtheorem{theorem}{Theorem}[section]
\newtheorem{corollary}[theorem]{Corollary}
\newtheorem{lemma}[theorem]{Lemma}

\newtheorem{proposition}[theorem]{Proposition}

\theoremstyle{definition}

\newtheorem{definition}[theorem]{Definition}
\newtheorem{remark}[theorem]{Remark}

\numberwithin{equation}{section}

\allowdisplaybreaks

\newcommand{\R}{\mathbb R}

\newcommand{\Z}{\mathbb Z}

\newcommand{\subRn}{{{\mathbb R}^n}}

\DeclareMathOperator{\dist}{dist}

\DeclareMathOperator*{\essinf}{ess\,inf}
\DeclareMathOperator*{\esssup}{ess\,sup}

\newcommand{\avf}{{\langle f\rangle}}
\newcommand{\avfa}{{\langle f_1\rangle}}
\newcommand{\avfb}{{\langle f_2\rangle}}

\newcommand{\avpafa}{{\langle f_1\rangle_{p_{1}(\cdot),Q}}}

\newcommand{\avpbfb}{{\langle f_2\rangle_{p_{2}(\cdot),Q}}}
\newcommand{\pap}{{p_{1}(\cdot)}}
\newcommand{\pbp}{{p_{2}(\cdot)}}
\newcommand{\cpap}{{p_{1}^{\prime}(\cdot)}}
\newcommand{\cpbp}{{p_{2}^{\prime}(\cdot)}}
\newcommand{\pjp}{{p_{j}(\cdot)}}
\newcommand{\cpjp}{{p_{j}^{\prime}(\cdot)}}
\newcommand{\pp}{{p(\cdot)}}
\newcommand{\cpp}{{p'(\cdot)}}
\newcommand{\Lp}{L^{p(\cdot)}}

\newcommand{\Lpa}{L^{p_{1}(\cdot)}}
\newcommand{\Lpb}{L^{p_{2}(\cdot)}}

\newcommand{\Pp}{\mathcal P}

\newcommand{\rr}{{r(\cdot)}}
\newcommand{\sst}{{s(\cdot)}}

\newcommand{\M}{\mathcal{M}}
\newcommand{\Q}{{Q_{j}^{k}}}
\newcommand{\vecpp}{{\vec{p}(\cdot)}}

\newcommand{\D}{\mathcal{D}}
\newcommand{\A}{\mathcal{A}}
\newcommand{\F}{\mathcal{F}}

\newcommand{\Qq}{{\mathcal{Q}}}

\newcommand{\sw}{{\mathcal{S}}}
\newcommand{\swp}{{\mathcal{S}'}}

\DeclareMathAlphabet\EuRoman{U}{eur}{m}{n}
\SetMathAlphabet\EuRoman{bold}{U}{eur}{b}{n}
\renewcommand{\mathsf}{\EuRoman}

\def\Xint#1{\mathchoice 
{\XXint\displaystyle\textstyle{#1}}%
{\XXint\textstyle\scriptstyle{#1}}%
{\XXint\scriptstyle\scriptscriptstyle{#1}}%
{\XXint\scriptscriptstyle\scriptscriptstyle{#1}}%
\!\int} 
\def\XXint#1#2#3{{\setbox0=\hbox{$#1{#2#3}{\int}$} 
\vcenter{\hbox{$#2#3$}}\kern-.5\wd0}}

\def\avgint{\Xint-}

\author{D.~Cruz-Uribe, OFS}
\address{Department of Mathematics \\ University of Alabama \\
Tuscaloosa, AL 35487, USA}
\email{dcruzuribe@ua.edu}

\author{O.~M.~Guzm\'an}
\address{Departamento de Matemáticas\\ Universidad Nacional de Colombia \\ AP360354 Bogotá \\ Colombia}
\email{omguzmanf@unal.edu.co}

\title[The bilinear maximal operator]{Weighted norm inequalities for
  the bilinear maximal operator on variable Lebesgue spaces}

\thanks{The first author is supported by 
  research funds from the Dean of the College of Arts \& Sciences, the
  University of Alabama.  This project was started while the second
  author spent the academic year 2016-17 visiting the University of
  Alabama.  The authors would like to thank the referees for their
  detailed and thorough reports.}

\date{August 8, 2019}

\begin{document}

\begin{abstract}
We extend the theory of weighted norm inequalities on variable
Lebesgue spaces to the case of bilinear operators.  We introduce a
bilinear version of the variable $\A_\pp$ condition, and show
that it is necessary and sufficient for the bilinear maximal operator
to satisfy a weighted norm inequality.  Our work generalizes the linear
results of the first author, Fiorenza and Neugebauer~\cite{dcu-f-nPreprint2010} in the variable
Lebesgue spaces and the bilinear results of Lerner {\em et al.}~\cite{MR2483720} in the
classical Lebesgue spaces.  As an application we prove weighted norm
inequalities for bilinear singular integral operators in the variable
Lebesgue spaces.
\end{abstract}

\maketitle


\section{Introduction}
\label{sec:intro}

In this paper we develop the theory of bilinear weighted norm
inequalities in the variable Lebesgue spaces.  To put our results in
context we will first describe some previous results; for brevity, we
will defer the majority of definitions until below.  The
Hardy-Littlewood maximal operator is defined by
\[ Mf(x) = \sup_Q \avgint_Q |f(y)|\,dy\cdot\chi_Q(x), \]
where the supremum is taken over all cubes in $\R^n$ with sides
parallel to the coordinate axes.  The now classical result of
Muckenhoupt~\cite{muckenhoupt72} is that a necessary and sufficient
condition for $M$ to be bounded on the weighted Lebesgue space
$L^p(w)$, $1<p<\infty$, i.e., that
\[ \int_\subRn (Mf)^pw\,dx \lesssim \int_\subRn |f|^p w\,dx, \]
 is that $w\in A_p$:
\[ \sup_Q \avgint_Q w\,dx \bigg(\avgint_Q w^{1-p'}\,dx\bigg)^{p-1} <
  \infty, \]
where again the  supremum is taken over all cubes in $\R^n$ with sides
parallel to the coordinate axes. 

This result has been generalized in two directions.  First, Lerner, {\em et
  al.}~\cite{MR2483720}, as part of the theory of weighted norm
inequalities for bilinear Calder\'on-Zygmund singular integrals,
introduced the bilinear (more properly, ``bisublinear'') maximal
operator:
\begin{equation*} \label{eqn:bilinear-max}
 \M(f_1,f_2)(x) = 
\sup_Q \avgint_Q |f_1(y)|\,dy \avgint_Q |f_2(y)|\,dy\cdot\chi_Q(x). 
\end{equation*}
It is immediate that $\M(f_1,f_2)(x)\leq Mf_1(x)Mf_2(x)$, and so by H\"older's
inequality,
\begin{equation} \label{eqn:bilinear-max-bound}
 \M : L^{p_1}(w_1) \times L^{p_2}(w_2) \rightarrow L^{p}(w), 
\end{equation}
where $1<p_1,\,p_2<\infty$, $\frac{1}{p}=\frac{1}{p_1}+\frac{1}{p_2}$,
$w_j\in A_{p_j}$, $j=1,2$, and $w=w_1^{\frac{p}{p_1}}w_2^{\frac{p}{p_2}}$. 

However, while this condition is sufficient, it is not necessary.  In
~\cite{MR2483720} they introduced the class $A_{\vec{p}}$ of vector weights
  defined as follows.  With the previous definitions, let
  $\vec{p}=(p_1,p_2,p)$ and let $\vec{w}=(w_1,w_2,w)$.  Then
  $\vec{w}\in A_{\vec{p}}$ if
\[ \sup_Q \bigg(\avgint_Q w\,dx\bigg)^{\frac{1}{p}}
\bigg(\avgint_Q w_1^{1-p_1'}\,dx\bigg)^{\frac{1}{p_1'}}
\bigg(\avgint_Q w_1^{1-p_2'}\,dx\bigg)^{\frac{1}{p_2'}}
< \infty. \]
They
proved that a necessary and sufficient condition for
inequality~\eqref{eqn:bilinear-max-bound} to hold is that $\vec{w}\in A_{\vec{p}}$.
If $w_j \in A_{p_j}$, then $\vec{w}\in A_{\vec{p}}$, but they gave
examples to show that the class $A_{\vec{p}}$ is strictly larger than
the weights gotten from $A_{p_1}\times A_{p_2}$.   

\smallskip

A second generalization of Muckenhoupt's result is to the setting of
the variable Lebesgue spaces.  The first author, Fiorenza and
Neugebauer~\cite{MR1976842} proved that given an exponent
function $\pp : \R^n\rightarrow [1,\infty)$ such that $1<p_-\leq
p_+<\infty$ and $\pp$ is log-H\"older continuous, then the maximal
operator is bounded on $\Lp$. 
Given $\pp$ a log-H\"older continuous function,  in~\cite{dcu-f-nPreprint2010} (see
also~\cite{cruz-diening-hasto2011}) they
proved the corresponding weighted norm inequality:  a necessary and
sufficient condition for the maximal operator to be bounded on
$\Lp(w)$, i.e., that $\|(Mf)w\|_\pp \lesssim \|fw\|_\pp$, is that
$w\in \A_\pp$,
\[ \sup_Q |Q|^{-1}\|w\chi_Q\|_\pp \|w^{-1}\chi_Q\|_\cpp < \infty. \]
When $\pp=p$ is a constant function, then this reduces to the
classical result of Muckenhoupt, since $\Lp(w)=L^p(w^p)$ and $w\in
\A_\pp$ is equivalent to $w^p \in A_p$.  

\smallskip

The purpose of this paper is to extend both of these results and
characterize the class of weights necessary and sufficient for the
bilinear maximal operator to satisfy bilinear weighted norm
inequalities over the variable Lebesgue spaces.  The remainder of this
paper is organized as follows.  In Section~\ref{section:main} we make
the necessary definitions to state our two main results; in
particular, we introduce the class of vector weights $\A_\vecpp$.  Our
first result,
Theorem~\ref{thm:main}, is for the bilinear maximal operator.  Our
second, Theorem~\ref{thm:sio}, shows that the weight condition
$\A_\vecpp$ is sufficient for bilinear Calder\'on-Zygmund singular
integral operators to satisfy weighted norm inequalities over the
variable Lebesgue spaces.  This generalizes the main result
of~\cite{MR2483720}.  

In Section~\ref{section:prelim} we gather some basic results about
weights and the variable Lebesgue spaces that we need in our proof,
and in Section~\ref{section:char} we prove  some properties of $\A_\pp$
and $\A_\vecpp$ weights.  In
Section~\ref{section:Ap-char} we give a characterization of
 vector weights $\A_\vecpp$  in terms of averaging operators.  In
Section~\ref{section:proof-main} we prove Theorem~\ref{thm:main}.  The
proof is broadly similar to the proof in the linear case given
in~\cite{dcu-f-nPreprint2010}, but there are many additional technical
obstacles.
Finally, in Section~\ref{section:proof-sio} we prove
Theorem~\ref{thm:sio}.  The proof relies on an extrapolation theorem
in the scale of weighted variable Lebesgue spaces proved
in~\cite{CruzUribeSFO:2017km}.

\begin{remark} {\em Added in proof.}  One
  of the anonymous referees for this paper asked whether a shorter
  proof of Theorem~\ref{thm:main} could be
  gotten by adapting the ideas of~\cite{cruz-diening-hasto2011} to the
  bilinear case.  We originally tried this approach, but were
  unsuccessful.  This remains an open problem.
\end{remark}

\medskip

Throughout this paper, $n$ will denote the dimension of the underlying
space $\R^n$.  A cube $Q\subset \R^n$ will always have its sides
parallel to the coordinate axes.  Let $\ell(Q)$ denote the side-length
of $Q$.  Given a cube $Q$ and a function $f$,
we will denote averages as follows:
\[ \frac{1}{|Q|}\int_Q f\,dx = \avgint_Q f\,dx = \avf_Q. \]
Similarly, if $\sigma$ is a non-negative measure, we denote weighted
averages by
\[ \frac{1}{\sigma(Q)}\int_Q f\sigma \,dx = \avf_{\sigma,Q}. \]

Constants will be denoted by $C$, $c$, etc.  and their value may
change from line to line, even in the same computation.  If we need to
emphasize the dependence of a constant on some parameter we will
write, for instance, $C(n)$.  Given two positive quantities $A$ and
$B$, we will write $A\lesssim B$ if there is a constant $c$ such that
$A\leq cB$.  If $A\lesssim B$ and $B\lesssim A$, then we write
$A\approx B$.  

\section{Main results}
\label{section:main}

We first recall the definition of variable Lebesgue spaces.  For more
information, see~\cite{cruz-fiorenza-book}.  Let $\Pp$ denote
the collection of measurable functions $\pp : \R^n \rightarrow
[1,\infty]$ and $\Pp_0$ the collection of measurable functions $\pp : \R^n \rightarrow
(0,\infty]$.  Given $\pp \in \Pp_0$ and a set $E\subset \R^n$, define
\[ p_-(E) = \essinf_{x\in E} p(x), \qquad p_+(E) = \esssup_{x\in E}
  p(x).  \]
For simplicity we will write $p_-=p_-(\R^n)$ and $p_+=p_+(\R^n)$.
Given $\pp\in \Pp$ we define the dual exponent $\cpp$ pointwise a.e.~by
\[ \frac{1}{p(x)}+\frac{1}{p'(x)}=1, \]
with the convention that $\frac{1}{\infty}=0$. 

The space $\Lp$ consists of all complex-valued, measurable functions $f$ such that for
some $\lambda>0$,
\[ \rho_\pp(f/\lambda) = \int_{\subRn\setminus \Omega_\infty}
  \bigg|\frac{f(x)}{\lambda}\bigg|^{p(x)}\,dx + \lambda^{-1}
  \|f\|_{L^\infty(\Omega_\infty)}
 < \infty, \]
where $\Omega_\infty = \{ x \in \R^n : p(x)=\infty \}$. 
This becomes a quasi-Banach function space when equipped with the norm 
\[ \|f\|_{\Lp}=\|f\|_\pp = \inf\big\{ \lambda>0 :
  \rho_\pp(f/\lambda)\leq 1 \big\};  \]
when $p_-\geq 1$ it is a Banach space.
When $\pp=p$, $0<p<\infty$, then $\Lp=L^p$ with equality of quasi-norms.

An exponent $\pp \in \Pp_0$ is said to be locally log-H\"older
continuous, denoted by $\pp \in LH_0$, if there exists a constant $C_0$
such that
\[ \bigg| \frac{1}{p(x)}-\frac{1}{p(y)}\bigg| \leq \frac{C_0}{-\log(|x-y|)}, \qquad
  |x-y|<\frac{1}{2}; \]
$\pp$ is said to be log-H\"older continuous at infinity, denoted by
$\pp \in LH_\infty$, if there exist $C_\infty,\,p_\infty>0$ such that
\[ \bigg|\frac{1}{p(x)} -\frac{1}{ p_\infty}\bigg| \leq \frac{C_\infty}{\log(e+|x|)}. \]
If $\pp \in LH = LH_0\cap LH_\infty$, we simply say that it is
log-H\"older continuous.  

\begin{remark}
For our main results we will assume $p_+<\infty$.  In this case
$\Omega_\infty$ has measure zero and the definition of the norm is simpler.
Moreover, in the definition of log-H\"older continuity, we can replace
the left-hand sides by $|p(x)-p(y)|$ and $|p(x)-p_\infty|$,
respectively, requiring only  new constants that depend on $p_+$.
\end{remark}

\medskip

By a weight $w$ we  mean a
non-negative function such that $0<w(x)<\infty$ a.e.  Given a weight
$w$ and $\pp\in \Pp_0$, we say $f\in \Lp(w)$ if $fw \in \Lp$.

\begin{definition}\label{Apcondition}
Given exponents
$\pap, \pbp \in \Pp(\subRn)$, define $\pp$  for a.e. $x$ by
\begin{equation} \label{eqn:p-defn}
 \frac{1}{p(x)}=\frac{1}{p_{1}(x)}+\frac{1}{p_{2}(x)}.
\end{equation}
and let $\vecpp = (\pap,\pbp,\pp)$.  Given weights
$w_1,\,w_2$, let $w=w_1w_2$ and define $\vec{w}=(w_1,w_2,w)$.  
We say that $\vec{w}\in \A_\vecpp$ if 
   \begin{equation*}
 \sup_Q |Q|^{-2}
\|w\chi_{Q}\|_{\pp}\|w_{1}^{-1}\chi_{Q}\|_{\cpap}
\|w_{2}^{-1}\chi_{Q}\|_{\cpbp}<\infty.
   \end{equation*}
\end{definition}  

\begin{remark}
If $\pap$ and $\pbp$ are constant, then this condition reduces to the
$A_{\vec{p}}\;$  condition for the triple
$(w_1^{p_1},w_2^{p_2},(w_1w_2)^p)$.  
\end{remark}

\medskip

We can now state our first main result.

\begin{theorem} \label{thm:main}
Given $\pap,\,\pbp \in \Pp$, suppose $1<(p_j)_-\leq (p_j)_+<\infty$ and
$p_j(\cdot)\in LH$, $j=1,\,2$. 
Define $\pp$ by \eqref{eqn:p-defn}.  Let $w_1,\,w_2$ be weights and
define $w=w_1w_2$.  Then the bilinear maximal operator satisfies
\begin{equation} \label{eqn:main1}
 \M : L^{\pap}(w_1) \times L^{\pbp}(w_2) \rightarrow L^{\pp}(w) 
\end{equation}
if and only if $\vec{w} \in \A_\vecpp$.  
\end{theorem}

\begin{remark}
  We do not believe that the assumption $\pp\in LH$ is necessary in
  Theorem~\ref{thm:main}, but some additional hypothesis is.  In the linear, unweighted case, while it is sufficient
  to assume that the exponent $\pp$ is log-H\"older continuous for the
  maximal operator to be bounded on $\Lp$, it is not necessary:
  see~\cite[Section 4.4]{cruz-fiorenza-book} for examples.  Diening and H\"ast\"o~\cite{diening-hastoPreprint2010}
 conjectured that in the weighted
  case, a necessary and sufficient condition for $M$ to be bounded on
  $\Lp(w)$ is that the maximal operator is bounded on $\Lp$ and $w\in
  \A_\pp$. (The latter condition is given in Definition
  \ref{Def:Ap-variable}.)  Unlike in the constant exponent case, when
  $w=1$ these conditions are not the same, since $1\in \A_\pp$ is a
  necessary but not sufficient condition for $M$ to be bounded on
  $\Lp$ \cite[Corollary~4.50, Example~4.51]{cruz-fiorenza-book}.  We conjecture that the analogous result holds in the
  bilinear case:  $\M$ satisfies \eqref{eqn:main1} if and only if $\M$
  satisfies an unweighted bilinear estimate and $\vec{w}\in \A_\vecpp$.  
\end{remark}

\begin{remark}
In the linear case, the maximal operator is bounded on $\Lp$ if
$p_->1$ and $1/\pp\in LH$:  we can allow $p_+=\infty$.
(See~\cite{cruz-fiorenza-book} for details and references.)   In~\cite{dcu-f-nPreprint2010} it
was conjectured that $M$ is bounded on $\Lp(w)$ with the same
hypotheses if $w\in \A_\pp$.  This condition is well defined even if
$p_-=1$ and $p_+=\infty$.  Moreover,  this conjecture is
true if $\pp=\infty$ a.e.  This is equivalent to a classical but often overlooked
result of Muckenhoupt~\cite{muckenhoupt72}, that if $w^{-1}\in A_1$
and $fw\in L^\infty$, then $(Mf)w\in L^\infty$.   Here we conjecture that
 we can remove the hypothesis $p_+<\infty$ from
Theorem~\ref{thm:main}.  However, as in the linear case we believe
that this will
require a very different argument, as the fact that $p_+,\,
(p_j)_+<\infty$ plays an important role in our proof.
\end{remark}

\begin{remark}
  Though we have only proved our result in the bilinear case, an
  $m$-linear version of Theorem~\ref{thm:main}, $m\geq 3$, should be
  true with the obvious changes in the definition of $\A_\vecpp$ and
  the statement of the theorem.  But even in the bilinear case the
  proof is quite technical, and so we decided to avoid making our
  proof even more obscure by trying to prove the general result.  
\end{remark}

\medskip

Our second main result is for bilinear Calder\'on-Zygmund
singular integrals.  These operators have been considered by a number
of authors, and we refer the reader to~\cite{MR2483720} for details
and further references.

Let $K(x,y,z)$ be a complex-valued, locally integrable function on
$\R^{3n}\setminus \triangle,$ where $\triangle=\{(x,x,x):x\in\R^n\}.$
$K$ is a Calder\'on-Zygmund kernel if there exist $A>0$ and $\delta>0$
such that for all $(x,y,z)\in \R^{3n}\setminus \triangle$,
\begin{equation*}\label{size}
\abs{K(x, y, z)} \le \frac{A}{(\abs{x-y}+\abs{x-z}+\abs{y-z})^{2n}}
\end{equation*}
and 
\begin{equation*}
\abs{K(x, y, z) - K(\tilde{x}, y, z)} \le  \frac{A\,\abs{x-\tilde{x}}^{\delta}}{(\abs{x-y}+\abs{x-z}+\abs{y-z})^{2n+\delta}}
\end{equation*}
whenever
$\abs{x - \tilde{x}} \leq \frac{1}{2} \max ( \abs{x-z}, \abs{x-y})$.
We also assume that the two analogous difference estimates with respect to the variables $y$ and
$z$ hold. An operator $T : \sw\times \sw \rightarrow\swp,$ is
a bilinear Calder\'on-Zygmund singular integral if:
\begin{enumerate}

\item there exists a bilinear Calder{\'o}n-Zygmund kernel
$K$ such that
    \begin{equation*}
    T(f_1,f_2)(x) = \int_{\R^{2n}} K(x, y, z) f_1(y) f_2(z) \,dy \, dz
    \end{equation*}
for all $f_1,\,f_2 \in C_c^\infty(\R^n)$ and all $x \notin \text{supp}(f_1) \cap \text{supp}(f_2);$

\item  there exist $1\le p,q<\infty$ and $r$ such that $\frac{1}{r}=\frac{1}{p}+\frac{1}{q}$ and $T$ can be extended to a  bounded operator from $L^{p} \times L^{q}$ into $L^r.$ 
\end{enumerate}

\begin{theorem} \label{thm:sio}
Given $\pap,\,\pbp \in \Pp$, suppose $1<(p_j)_-\leq (p_j)_+<\infty$ and
$p_j(\cdot)\in LH$, $j=1,\,2$. 
Define $\pp$ by \eqref{eqn:p-defn}.  Let $w_1,\,w_2$ be weights,
define $w=w_1w_2$, and assume $\vec{w} \in \A_\vecpp$.  If $T$ is a bilinear Calder\'on-Zygmund singular
integral, then
\begin{equation} \label{eqn:sio1}
 T : L^{\pap}(w_1) \times L^{\pbp}(w_2) \rightarrow L^{\pp}(w).
\end{equation}
\end{theorem}

\begin{remark}
As for the bilinear maximal operator, we do not believe that the
assumption that $\pap,\,\pbp \in LH$ is necessary for the conclusion
in Theorem~\ref{thm:sio} to hold.
In~\cite{CruzUribe:2016wv}, the authors proved that in the unweighted
case it was sufficient to assume that the (linear) maximal operator is
bounded on $L^\pap$ and $L^\pbp$.  We conjecture that with this
hypothesis, or even the weaker assumption that $\M$ satisfies the
associated unweighted bilinear inequality, and $\vec{w}\in \A_\vecpp$,
then \eqref{eqn:sio1} holds.  
\end{remark}

\begin{remark}
Alongside the variable Lebesgue spaces there is a theory of variable
Hardy spaces:  see~\cite{DCU-dw-P2014}.  
  Very recently, the first author, Moen and Nguyen~\cite{DCU-KM-HN}
  proved unweighted estimates on variable Hardy spaces for bilinear
  Calder\'on-Zygmund singular integrals.  It would be interesting to
  extend these results to weighted variable Hardy spaces using the
  $\A_\vecpp$ weights.  
\end{remark}

\section{Preliminaries}
\label{section:prelim}

In this section we gather some basic results about weights and about
variable Lebesgue spaces that we will need in the
subsequent sections.   

\subsection*{Weights}
First, we recall the definition of the class $A_\infty$:
\[ A_\infty = \bigcup_{p>1} A_p.  \]
We will need the following property of $A_\infty$ weights.  For a
proof, see~\cite{garcia-cuerva-rubiodefrancia85}.   

\begin{lemma} \label{lemma:Ainfty-prop}
Let $w\in A_\infty$.  Then for each $0<\alpha<1$ there exists $0<\beta <1$ such
that if $Q$ is any cube and $E\subset Q$ is such that $\alpha|Q|\leq
|E|$, then $\beta w(Q)\leq  w(E)$.   Similarly, for each $0<\gamma<1$
there exists $0<\delta<1$ such that if $|E|\leq \gamma |Q|$, then
$w(E) \leq \delta w(Q)$.  
\end{lemma}

To state our next result,  we introduce the weighted dyadic maximal
operator.  Given a weight $\sigma$, 
\[ M_\sigma^{\D_0} f(x) = \sup_{Q\in \D_0} 
\frac{1}{\sigma(Q)}\int_Q |f|\sigma\,dy\cdot \chi_Q(x) 
= \sup_{Q\in \D_0} \langle |f|\rangle_{\sigma,Q} \chi_Q(x), \]
where the supremum is taken over all cubes in the collection $\D_0$ of
dyadic cubes:
\[ \D_0 = \{ 2^{-k}([0,1)^n +j) :  k \in \Z, j \in \Z^n \}. \]
The following result is well-known but an explicit proof does not seem
to have appeared in the literature.  The proof is essentially the same
as for the classical dyadic maximal operator:  see
Grafakos~\cite{grafakos08a}.  

\begin{lemma} \label{lemma:wtd-max-bound}
Given a weight $\sigma$, then for $1<p<\infty$,
    \[ 
        \int_{\subRn} \big( M_{\sigma}^{\D_0} f\big)^{p}\sigma \,dx \lesssim \int_{\subRn} |f|^{p}\sigma \,dx    \]
 and the implicit constant depends only
on $p$.
\end{lemma}

\subsection*{Variable Lebesgue spaces}
Here we gather some basic results about variable Lebesgue spaces.  All
of these are found in the literature (with some minor variations).  In
some cases they were only proved for exponents $\pp \in \Pp$, but
essentially the same proof works for $\pp \in \Pp_0$.

\begin{lemma} \label{lemma:rescale}
\cite[Proposition~2.18]{cruz-fiorenza-book}
 Given $\pp \in \Pp_0$, suppose  $|\Omega_\infty|=0$. If $s>0$,  then
\[ \||f|^s\|_\pp = \|f\|_{s\pp}^s.  \]
\end{lemma}

\begin{lemma} \label{lemma:fatou}
\cite[Theorem~2.61]{cruz-fiorenza-book}
Given $\pp\in \Pp_0$, if $f\in \Lp$ is such that $\{f_k\}$ converges
to $f$ pointwise a.e., then
\[ 
\|f\|_\pp \leq \liminf_{k\rightarrow \infty} \|f_k\|_\pp \]
\end{lemma}

\begin{lemma} \label{lemma:mod-norm}
\cite[Corollary~2.23]{cruz-fiorenza-book}
Given $\pp \in \Pp_0$, suppose $0<p_-\leq p_+<\infty$.  If
$\|f\|_\pp>1$, then
\[ \rho_\pp(f)^{\frac{1}{p_+}}\leq \|f\|_\pp \leq 
\rho_\pp(f)^{\frac{1}{p_-}}. \]
If
$\|f\|_\pp\leq 1$, then
\[ \rho_\pp(f)^{\frac{1}{p_-}}\leq \|f\|_\pp \leq 
\rho_\pp(f)^{\frac{1}{p_+}}. \]
Consequently, $\|f\|_\pp\lesssim 1$ if and only if
$\rho_\pp(f)\lesssim 1$. 
\end{lemma}

\begin{lemma} \label{lemma:gen-holder}
\cite[Corollary~2.30]{cruz-fiorenza-book}
Fix $k\geq 2$ and let $p_j(\cdot)\in \Pp$ satisfy for a.e. $x$,
\[ \sum_{j=1}^k \frac{1}{p_j(x)} = 1.  \]
Then, for all $f_{j}\in L^{p_{j}(\cdot)}$, $1\leqslant j\leqslant k$,
\[ \int_\subRn |f_1\cdots f_k|\,dx 
\lesssim \prod_{j=1}^k \|f_j\|_{p_j(\cdot)}.  \]
The implicit constant depends only on the $p_j(\cdot)$.
\end{lemma}

\begin{lemma} \label{lemma:holder}
\cite[Corollary~2.28]{cruz-fiorenza-book}
Given $\pap,\,\pbp \in \Pp_0$, define $\pp\in \Pp_0$
by~\eqref{eqn:p-defn}.  Then 
\[ \|fg\|_\pp \lesssim \|f\|_\pap \|g\|_\pbp; \]
the implicit constant depends only on $\pap$ and $\pbp$.
\end{lemma}

\begin{lemma} \label{lemma:dual}
	\cite[Theorem~2.34]{cruz-fiorenza-book}
	Given $\pp \in \Pp$, then for every $f\in \Lp$,
	\[ \|f\|_\pp \approx \sup_{\|g\|_\cpp \leq
		1}\int_\subRn |fg|\,dx. \]
	The implicit constants depend only on $\pp$.
\end{lemma}

\begin{remark}
It is immediate the in the weighted space $L^\pp(w)$, the same result
is true if we take the supremum over all $g\in L^\cpp(w^{-1})$ with 
$\|gw^{-1}\|_\cpp \leq 1$.  
\end{remark}

\begin{lemma} \label{lemma:harmonic}
\cite[Corollary~4.5.9]{diening-harjulehto-hasto-ruzicka2010}
Let $\pp \in \Pp$ be such that $\pp \in LH$.  Then for
every cube $Q$,
\[ \|\chi_Q\|_\pp \approx |Q|^{\frac{1}{p_Q}}, \]
where $p_Q$ is the harmonic mean of $\pp$ on $Q$:
\[ \frac{1}{p_Q} = \avgint_Q \frac{dx}{p(x)}. \]
The implicit constants depend only on $\pp$.
\end{lemma}

\begin{lemma} \label{lemma:diening}
\cite[Lemma~3.24]{cruz-fiorenza-book}
Given $\pp \in \Pp_0$, suppose $\pp\in LH_0$ and $0<p_-\leq p_+<\infty$.
Then for every cube $Q$,
\[ |Q|^{p_-(Q)-p_+(Q)} \lesssim 1, \]
and the implicit constant depends only on $\pp$ and $n$.  The same
inequality holds if we replace one of $p_+(Q)$ or $p_-(Q)$ by $p(x)$ for
a.e. $x\in Q$.  
\end{lemma}

\begin{remark}
Lemma~\ref{lemma:diening} is sometimes referred to as Diening's
condition, and it is the principal way in which we will apply the
$LH_0$ condition. 
\end{remark}

\begin{lemma} \label{lemma:p-infty-px}
\cite[Lemmas~2.7,~2.8]{capone-cruz-fio} 
Given two exponents $\rr,\,\sst \in\Pp_0$, suppose 
\[ |s(y)-r(y)| \leq \frac{C_0}{\log(e+|y|)}. \]
Then given any set $G$ and any non-negative measure $\mu$, for every
$t\geq 1$ there exists a constant $C=C(t,C_0)$ such that for all
functions $f$ such that $|f(y)|\leq 1$,
\[ \int_Q |f(y)|^{s(y)}\,d\mu(y) 
\leq C\int_G |f(y)|^{r(y)}\,d\mu(y)
+ \int_G \frac{1}{(e+|y|)^{tns_-(G)}}\,d\mu(y).  \]
If we instead assume that 
\[  0 \leq r(y)-s(y) \leq \frac{C_0}{\log(e+|y|)}, \]
then the same inequality holds for any function $f$.    
\end{lemma}

\begin{remark}
Lemma~\ref{lemma:p-infty-px} is the principal way in which we will apply
the $LH_\infty$ condition.  
\end{remark}

\section{Properties of $\A_\pp$ and $\A_{\vecpp}$ weights}
\label{section:char}

In this section we give some properties of the $\A_\pp$ and $\A_\vecpp$
weights that will be used in the proof of Theorem~\ref{thm:main}.
 For simplicity, throughout this section, assume
that $w_1,\,w_2$ are weights and let $w=w_1w_2$ and
$\vec{w}=(w_1,w_2,w)$.  Similarly, whenever we are given
$\pap,\,\pbp \in \Pp$, define $\pp$ by~\eqref{eqn:p-defn} and let
$\vecpp=(\pap,\pbp,\pp)$.  Note that in this case we always have that
$p_-\geq \frac{1}{2}$.

We begin by recalling the definition of $\A_\pp$ weights and then state
several results from~\cite{dcu-f-nPreprint2010} on their
properties.  

\begin{definition}\label{Def:Ap-variable}
Given $\pp \in \Pp$ and a weight $w$, we say $w\in \A_\pp$ if
\[ \sup_Q |Q|^{-1}\|w\chi_Q \|_\pp \|w^{-1}\chi_Q \|_{\cpp} < \infty. \]
\end{definition}

The next two lemmas show the relationship between $\A_\pp$ and
$A_\infty$ weights.

\begin{lemma} \label{lemma:App-Ainfty}
\cite[Lemma~3.4]{dcu-f-nPreprint2010}
Given $\pp \in \Pp$, suppose $\pp \in LH$ and $p_+<\infty$.  If $w\in
\A_\pp$, then $u(\cdot) = w(\cdot)^\pp \in A_\infty$. 
\end{lemma}

\begin{lemma} \label{lemma:p-infty-cond}
  \cite[Lemmas~3.5,~3.6]{dcu-f-nPreprint2010} Given $\pp \in \Pp$,
  suppose $\pp \in LH$ and $p_+<\infty$.  Let $w\in \A_\pp$ and let
  $u(x)=w(x)^{p(x)}$.  Then given any cube $Q$ such that
  $\|w\chi_Q\|_\pp \geq 1$, then
  $\|w\chi_Q\|_\pp \approx u(Q)^{\frac{1}{p_\infty}}$.  Moreover,
  given any $E\subset Q$,
\[ \frac{|E|}{|Q|} \lesssim
  \bigg(\frac{u(E)}{u(Q)}\bigg)^{\frac{1}{p_\infty}}. \]
The implicit constants depend only on $w$ and $\pp$.
\end{lemma}

\begin{remark}
To apply Lemma~\ref{lemma:p-infty-cond}, note that by
Lemma~\ref{lemma:mod-norm}, $\|w\chi_Q\|_\pp \geq 1$ if and only if
$u(Q) \geq 1$. 
\end{remark}

The next result is a weighted version of the Diening
condition in Lemma~\ref{lemma:diening} and will be used to apply the
$LH_0$ condition.  

\begin{lemma} \label{lemma:lemma3.3}
\cite[Lemma~3.3]{dcu-f-nPreprint2010}
Given $\pp\in \Pp$ such that $\pp\in LH$, if $w\in \A_\pp$, then for
all cubes $Q$,
\[ \|w\chi_Q\|_\pp^{p_-(Q)-p_+(Q)} \lesssim 1; \]
the implicit constant depends only on $\pp$ and $w$.
\end{lemma}

The final lemma is an integral estimate that, in conjunction with
Lemma~\ref{lemma:p-infty-px} will be used to apply the $LH_\infty$
condition. 

\begin{lemma} \label{lemma:infty-bound}
\cite[Inequality~(3.3)]{dcu-f-nPreprint2010}
Given $\pp \in \Pp$, suppose $\pp \in LH$.  If $w\in \A_\pp$, then
there exists a constant $t>1$, depending only on $w$, $\pp$ and $n$,
such that
\[ \int_\subRn \frac{w(x)^{p(x)}}{(e+|x|)^{tnp_-}}\,dx \leq 1. \]
\end{lemma}

\medskip

We now turn to the $\A_\vecpp$ condition.  
If $w_1\in \A_\pap$ and $w_2 \in \A_\pbp$, then by Lemma~\ref{lemma:gen-holder} we have
that $\vec{w}\in \A_\vecpp$.  However, this inclusion is proper, since
it is in the constant exponent case.  Nevertheless,
we can  characterize the bilinear $\A_\vecpp$ weights in terms
of the $\A_\pp$ condition.   In the constant exponent case this is
proved in~\cite{MR2483720}, and our argument is adapted from theirs.

\begin{proposition}\label{AP-charachterization}
Given $\vec{w}$ and $\vecpp$, $\vec{w} \in \mathcal{A}_{\vec{p}(\cdot)}$ if and only if 
    \begin{equation}\label{EQ-characterizationApcondition}
      \begin{cases}
          w_{j}^{-\frac{1}{2}} \in \mathcal{A}_{2\cpjp} , & j=1,2\\
          w^{\frac{1}{2}}\in \mathcal{A}_{2\pp}, &
       \end{cases}
    \end{equation}
\end{proposition}

\begin{remark}
Note that since $p_- \geq\frac{1}{2}$, 
 $2\pp \in \Pp$ and $\A_{2\pp}$ is well defined.
\end{remark}

\begin{proof}
First assume that $\vec{w}\in \A_\vecpp$. Then for
a.e. $x$,
\begin{multline*}
 \frac{1}{2p(x)}+\frac{1}{2p_{2}^{\prime}(x)}
= 1-\frac{1}{(2p)^{\prime}(x)}+\frac{1}{2p_{2}^{\prime}(x)}\\
= 1-\left(\frac{1}{2p_{1}^{\prime}(x)}+\frac{1}{2p_{2}^{\prime}(x)}\right)
+ \frac{1}{2p_{2}^{\prime}(x)}
= 1-\frac{1}{2p_{1}^{\prime}(x)}
= \frac{1}{(2p_{1}^{\prime})^{\prime}(x)}. 
\end{multline*}
Therefore, by Lemmas~\ref{lemma:holder} and~\ref{lemma:rescale}, and by
  the definition of $\A_\vecpp$,
   \begin{align*}
  \|w_{1}^{-\frac{1}{2}}\chi_{Q}\|_{2\cpap}\|w_{1}^{\frac{1}{2}}\chi_{Q}\|_{(2\cpap)^{\prime}}
& =  \|w_{1}^{-\frac{1}{2}}\chi_{Q}\|_{2\cpap}
\|w_{1}^{\frac{1}{2}}w_{2}^{\frac{1}{2}}w_{2}^{-\frac{1}{2}}\chi_Q \|_{(2p_{1}^{\prime})^{\prime}(\cdot)} \\
    &\lesssim \|w_{1}^{-\frac{1}{2}}\chi_{Q}\|_{2\cpap}
\|w^{\frac{1}{2}}\chi_{Q}\|_{2\pp}
\|w_{2}^{-\frac{1}{2}}\chi_{Q}\|_{2\cpbp}\\
    &= \bigg(
      \|w\chi_{Q}\|_{\pp}\prod_{j=1}^{2}{\|w_{j}^{-1}\chi_Q \|_{\cpjp}} 
\bigg)^{\frac{1}{2}}\\
    &\lesssim \abs{Q}.
   \end{align*}
Hence, $w_{1}^{-\frac{1}{2}} \in \mathcal{A}_{2\cpap}$.  The same
argument shows that $w_{2}^{-\frac{1}{2}} \in \mathcal{A}_{2\cpbp}$.
Finally, we have that
\begin{align*}
\|w^{\frac{1}{2}}\chi_{Q}\|_{2\pp}\|w^{-\frac{1}{2}}\chi_{Q}\|_{(2p)^{\prime}(\cdot)}
&\lesssim  \|w^{\frac{1}{2}}\chi_{Q}\|_{2\pp}
\prod_{j=1}^{2}\|w_{j}^{-\frac{1}{2}}\chi_{Q}\|_{2\cpjp} \\
&= \bigg(
\|w\chi_Q \|_{\pp}\prod_{j=1}^{2}{\|w_{j}^{-1}\chi_Q \|_{\cpjp}}
\bigg)^{\frac{1}{2}} 
\lesssim \abs{Q}.
\end{align*}
Thus~\eqref{EQ-characterizationApcondition} holds.

\smallskip

 Conversely, now suppose that \eqref{EQ-characterizationApcondition}
 holds. Then for a.e. $x$, 
 \begin{equation*}
       \frac{1}{2(2p)^{\prime}(x)}
+\frac{1}{2(2p_{1}^{\prime})^{\prime}(x)}
+\frac{1}{2(2p_{2}^{\prime})^{\prime}(x)}=1,
       \end{equation*}
so by Lemmas~\ref{lemma:gen-holder} and \ref{lemma:rescale},  for any
cube $Q$,
  \begin{align*}
  1
  = \langle w^{-\frac{1}{4}}w_{1}^{\frac{1}{4}}w_{2}^{\frac{1}{4}} \rangle_{Q}^{2}
  &\lesssim
  \abs{Q}^{-2}\|w^{-\frac{1}{4}}\chi_Q \|_{2(2p)^{\prime}(\cdot)}^{2}
  \prod_{j=1}^{2}\|w_{j}^{\frac{1}{4}}\chi_{Q}\|_{2(2p_{j}^{\prime})^{\prime}(\cdot)}^{2} \\
  &= \abs{Q}^{-2}
  \|w^{-\frac{1}{2}}\chi_Q\|_{(2p)^{\prime}(\cdot)}
  \prod_{j=1}^{2}\| w_{j}^{\frac{1}{2}}\chi_Q\|_{(2p_{j}^{\prime})^{\prime}(\cdot)}.
  \end{align*}
Therefore,
   \begin{multline*}
\|w \chi_Q\|_{\pp}^{\frac{1}{2}}\prod_{j=1}^{2}
\|w_{j}^{-1}\chi_Q\|_{\cpjp}^{\frac{1}{2}} 
= \|w^{\frac{1}{2}}\chi_Q\|_{2\pp}
\prod_{j=1}^{2}\|w_{j}^{-\frac{1}{2}}\chi_Q\|_{2\cpjp}\\
\lesssim \abs{Q}^{-2} \|w^{\frac{1}{2}}\chi_Q\|_{2\pp}
\|w^{-\frac{1}{2}}\chi_Q\|_{(2p)^{\prime}(\cdot)}
\prod_{j=1}^{2}\|
w_{j}^{\frac{1}{2}}\chi_Q\|_{(2p_{j}^{\prime})^{\prime}(\cdot)}
\|w_{j}^{-\frac{1}{2}}\chi_Q\|_{2\cpjp}
 \lesssim \abs{Q},
\end{multline*}
and so   $w \in \mathcal{A}_{\vec{p}(\cdot)}$ .   
\end{proof}

Proposition~\ref{AP-charachterization} has the following corollary
which will be used in our proof of Theorem~\ref{thm:main}.  

\begin{corollary} \label{cor:Ainfty}
Given $\pap,\,\pbp \in \Pp$, suppose $p_j(\cdot)\in LH$ and
$(p_j)_+<\infty$, $j=1,2$.  If
$\vec{w}\in \A_\vecpp$, then $u(\cdot)=w(\cdot)^\pp$ and
$\sigma_j(\cdot)=w_j(\cdot)^{-p_j'(\cdot)}$, $j=1,2$, are in $A_\infty$.  
\end{corollary}

\begin{proof}
  This follows immediately from Lemma~\ref{lemma:App-Ainfty} and
  Proposition~\ref{AP-charachterization}.
\end{proof}

Our next lemma is a variant of Lemma~\ref{lemma:lemma3.3} for
$\A_\vecpp$ weights.   The proof is adapted from the proof
in~\cite{dcu-f-nPreprint2010}. 

\begin{proposition}\label{q-relation}
Given $\pap,\,\pbp \in \Pp$, suppose $p_j(\cdot)\in LH$, $j=1,2$.
Define $\pp$ by~\eqref{eqn:p-defn} and suppose $p_+<\infty$.   For
every cube $Q$, define $q(Q)$ by

   \begin{equation*}
   \frac{1}{q(Q)}=\frac{1}{(p_{1})_{-}(Q)}+\frac{1}{(p_{2})_{-}(Q)}.
   \end{equation*}
Then, given  $v\in \A_\pp$,  for a.e. $x\in Q$,
\begin{equation}\label{EQ-keyboundednorm}
            \|v^{-1}\chi_Q\|_{\cpp}^{q(Q)-p(x)}\lesssim 1;
\end{equation}               
the implicit constant depends on $\pap,\,\pbp$, $n$ and $v$, but is independent of $Q$.
\end{proposition}

\begin{remark}
  When we apply Proposition~\ref{q-relation} below, we will do so in
  conjunction with Proposition~\ref{AP-charachterization} to
  $w^{\frac{1}{2}}\in \A_{2\pp}$, so we will let
  $v^{-1}=w^{-\frac{1}{2}}$ and replace $\pp$ by $2\pp$ and $q$ by
  $2q$.  We could have stated this result in these terms, but for the
  purposes of the proof, it seemed easier to suppress the factor of 2.
\end{remark}

\begin{proof}
Fix a cube $Q \subset \subRn$.   It follows from the definition that for a.e. $x\in Q$,
 $q(Q)\leq p(x)\leq p_+$, so if
$\|v^{-1}\chi_Q\|_\cpp>1$, \eqref{EQ-keyboundednorm} holds
immediately.  Therefore, we may assume without loss of generality that
$\|v^{-1}\chi_Q\|_\cpp\leq 1$. 

Let $Q_{0}$ be the cube centered at the origin with $|Q_0|=1$.  Then
either $|Q|\leq |Q_0|$ or $|Q|>|Q_0|$.  We will
prove~\eqref{EQ-keyboundednorm} in the first case; the proof of the
second case is the same, exchanging the roles of $Q$ and $Q_0$. Let
$\ell(Q_{0})$ be the side-length
of the cube $Q_0$ and suppose  that  $\dist(Q,Q_{0})\leqslant \ell(Q_{0})$; then
$Q \subset 5Q_{0}$.   Define
\begin{equation} \label{eqn:q-minus}
 q_- = \inf_Q q(Q)\leq \sup_Q q(Q) \leq p_+<\infty.  
\end{equation}
Then
\[  \frac{1}{q(Q)}-\frac{1}{p(x)}
\leq
\left(\frac{1}{(p_{1})_{-}(Q)}-\frac{1}{(p_{1})_{+}(Q)}\right)
+\left(\frac{1}{(p_{2})_{-}(Q)}-\frac{1}{(p_{2})_{+}(Q)}\right), \]
so there exists a constant $C=C(\pap,\pbp)$ such that
\begin{equation}\label{EQ-logcontinuityof-q}
     p(x)-q(Q)\leq C\left[(p_{1})_{+}(Q)-(p_{1})_{-}(Q)\right]+C\left[(p_{2})_{+}(Q)-(p_{2})_{-}(Q)\right].
\end{equation}

Therefore, by Lemma~\ref{lemma:gen-holder} and the
$\mathcal{A}_{\cpp}$ condition we have that
\begin{multline} \label{eqn:small}
\abs{Q} = \int_Q v^{-1}v\,dx
\lesssim  \| v^{-1}\chi_Q\|_{\cpp}\| v \chi_Q\|_{\pp}\\
       \leq 5^{n}\| v^{-1}\chi_Q\|_{\cpp}\abs{5Q_{0}}^{-1}\|v\chi_{5Q_{0}}\|_{\pp}
        \lesssim \|v^{-1}\chi_Q
        \|_{\cpp}\|v^{-1}\chi_{5Q_{0}}\|_{\cpp}^{-1}.  
\end{multline}
Hence, by \eqref{EQ-logcontinuityof-q} and Lemma~\ref{lemma:diening},
 \begin{multline*}
\|v^{-1}\chi_Q\|_{\cpp}^{q(Q)-p(x)}
\lesssim \|v^{-1}\chi_{5Q_{0}}\|_{\cpp}^{q(Q)-p(x)}\abs{Q}^{q(Q)-p(x)}\\
 \leq \left(1+\|v^{-1}\chi_{5Q_{0}}\|_{\cpp}^{-1}\right)^{p_{+}-q_{-}}\abs{Q}^{q(Q)-p(x)}
\lesssim 1.
 \end{multline*}
 
\medskip

Now assume that $\dist(Q,Q_{0})\geqslant \ell(Q_{0})$.  Then there exists a cube
$\hat{Q}$ such that $Q,\, Q_0 \subset \hat{Q}$ and
$\ell(\hat{Q})\approx \dist(Q,Q_{0})\approx
\dist(Q,0)=d_{Q}$. Therefore, arguing as we did in
inequality~\eqref{eqn:small}, replacing $5Q_0$ by $\hat{Q}$, we get
   \begin{equation*}
   \abs{Q} \lesssim\vert\hat{Q}\vert\|v^{-1}\chi_Q\|_{\cpp}\|v^{-1}\chi_{\hat{Q}}\|_{\cpp}^{-1}.
   \end{equation*}
If we continue the above argument and use the fact that 
$\|v^{-1}\chi_{Q_0}\|_{\cpp} \leq \|v^{-1}\chi_{\hat{Q}}\|_{\cpp}$, we get
\[\|v^{-1}\chi_Q \|_{\cpp}^{q(Q)-p(x)}\lesssim
  \vert\hat{Q}\vert^{p(x)-q(Q)}. \]

To estimate this final term, note that since $p_j(\cdot)\in LH$, there
exist $x_{1}, x_{2} \in \overline{Q}$ such that $(p_{1})_{-}(Q)=p_{1}(x_{1})$
and $(p_{2})_{-}(Q)=p_{2}(x_{2})$.   Moreover,
$\abs{x_{1}},\abs{x_{2}}\approx d_{Q}$.  Therefore, again by
log-H\"older continuity, and using that 
$\frac{1}{p_\infty}=\frac{1}{(p_1)_\infty}+\frac{1}{(p_2)_\infty}$,
\[ 
\abs{\frac{1}{q(Q)}-\frac{1}{p_{\infty}}}
\leq
\left|\frac{1}{p_1(x_1)}-\frac{1}{(p_1)_\infty}\right|
+
\left|\frac{1}{p_2(x_2)}-\frac{1}{(p_2)_\infty}\right|
\lesssim 
\frac{1}{\log(e+d_{Q})}. \] 
Therefore, for $x\in Q$, since $\abs{x}\approx d_{Q}$,
  \[        \abs{\frac{1}{q(Q)}-\frac{1}{p(x)}}
\leq \abs{\frac{1}{q(Q)}-\frac{1}{p_{\infty}}}+\abs{\frac{1}{p_{\infty}}-\frac{1}{p(x)}}
\lesssim\frac{1}{\log(e+d_{Q})},
\]
 Given this, and since  $\vert\hat{Q}\vert \lesssim (e+d_{Q})^{n}$, we
 thus have that
 \begin{equation*}
 \vert\hat{Q}\vert^{p(x)-q(Q)}\lesssim 1, 
 \end{equation*}
and so  $\| v^{-1}\chi_Q\|_{\cpp}^{q(Q)-p(x)}\lesssim 1$.
  \end{proof}

\section{Characterization of $\A_\vecpp$}
\label{section:Ap-char}

In this section we give two characterizations of the $\A_\vecpp$
condition in terms of averaging operators.  The first is a very
general condition that does not require assuming that the exponent
functions are log-H\"older continuous; the second requires the
additional assumption that $\pap,\,\pbp$ are log-H\"older continuous.

Given $Q$ a cube, define the multilinear averaging operator
${A}_{Q}$ by
\begin{equation*}
  A_{Q}(f_1,f_2)(x):=\avfa_{Q}\avfb_{Q}\chi_Q (x).
\end{equation*} 
More generally, given a family $\Qq=\{Q\}$ of disjoint cubes, we define 
\[ T_\Qq(f_1,f_2)(x) = \sum_{Q\in \Qq} A_Q(f_1,f_2)(x)\chi_Q(x).  \]

\begin{theorem} \label{thm:vecAp-char}
Given $\pap,\,\pbp \in \Pp$ and $\vec{w}$, then $\vec{w}\in \A_\vecpp$ if and only if
\begin{equation}\label{AQequation}
\sup_{Q}\|{A}_{Q}(f_1,f_2)w\|_{\pp}
\lesssim \|f_1w_{1}\|_{\pap}\|f_2w_{2}\|_{\pbp},
\end{equation}
where the supremum is taken over all cubes $Q$.
If we assume further that $\pap,\,\pbp \in LH$, then $\vec{w}\in
\A_\vecpp$ if and only if 
\begin{equation}\label{TQequation}
\sup_{\Qq}\|{T}_{\Qq}(f_1,f_2)w\|_{\pp}
\lesssim \|f_1w_{1}\|_{\pap}\|f_2w_{2}\|_{\pbp},
\end{equation}
where the supremum is taken over all collections $\Qq$ of disjoint
cubes. 
\end{theorem}

\begin{remark}
When $p_-\geq 1$ (i.e., when $\Lp$ is a Banach space) the
characterization in terms of the operators $T_\Qq$ is a consequence of
a general result in the setting of Banach lattices due to
Kokilashvili, {\em et al.}~\cite{Kokilashvili:2015gw}.  However, even
in this special case we would be required to show that condition
$\vec{G}$ defined below holds in order to apply their result.  In our
case we can use the rescaling properties of variable Lebesgue spaces
to prove it directly.
\end{remark}

\begin{remark}
A very deep result in the theory of variable Lebesgue spaces is that
the uniform boundedness of the linear version of the averaging operators
$T_\Qq$ is equivalent to the boundedness of the Hardy-Littlewood
maximal operator, but the uniform boundedness of the (linear)
operators $A_Q$ is not.
See ~\cite[Section 4.4]{cruz-fiorenza-book},~\cite{Diening2005} and~\cite[Section 5.2]{diening-harjulehto-hasto-ruzicka2010} for
details and further references.  We conjecture that the corresponding
result holds in the bilinear case.
\end{remark}

The proof of Theorem~\ref{thm:vecAp-char} is straightforward for
$A_Q$, and so we give this proof separately.

\begin{proof}[Proof of Theorem~\ref{thm:vecAp-char} for $A_Q$]
Let be $w \in \mathcal{A}_{\vec{p}(\cdot)}$; then given any cube $Q$,
by Lemma~\ref{lemma:gen-holder} and the definition of
$\mathcal{A}_{\vec{p}(\cdot)}$  we get 
\begin{align*}
  \|A_{Q}(f_1,f_2)w\|_{\pp} 
&= |Q|^{-2}\int_{Q}{|f_1|w_{1}w_{1}^{-1}\,dy}
\int_{Q}{|f_2|w_{2}w_{2}^{-1}\,dy} \,\|w\chi_Q\|_{\pp}\\
&\lesssim |Q|^{-2}\|w_{1}^{-1}\chi_Q\|_{\cpap}
\| w_{2}^{-1}\chi_Q\|_{\cpbp}
\|w\chi_Q\|_{\pp}\|f_1w_{1}\|_{\pap}\|f_2w_{2}\|_{\pbp}\\
  & \lesssim \|f_1w_{1}\|_{\pap}\|f_2w_{2}\|_{\pbp}.
\end{align*}
Since the implicit constant depends only on the $\A_\vecpp$ condition
and is independent of $Q$, we get~\eqref{AQequation}.

\medskip
  
Now assume that \eqref{AQequation} holds. By Lemma~\ref{lemma:dual},
there exist $h_{j}w_{j}\in L^{\pjp}$, 
$\|h_{j}w_{j}\|_{\pjp}\leq 1$, $j=1,2$, such that
 \begin{align*}
\|w\chi_Q\|_{\pp}
\prod_{j=1}^{2}\| w_{j}^{-1}\chi_{Q}\|_{\cpjp} 
& \lesssim 
\|w\chi_Q\|_{\pp}\int_Q h_{1}\,dy\int_Q h_{2}\,dy\\
& = \|w\chi_Q\|_{\pp} 
\left\langle h_{1}\right\rangle_{Q}\left\langle h_{2}\right\rangle_{Q}
|Q|^{2} \\
& = \|A_{Q}(h_{1},h_{2})w\|_{\pp}|Q|^{2} \\
&  \lesssim \|h_{1}w_{1}\|_{\pap}\|h_{2}w_{2}\|_{\pbp}|Q|^2\\
&  \lesssim |Q|^{2}.   
 \end{align*}
Again, the constant is independent of $Q$, so $\vec{w}\in \A_\vecpp$.
\end{proof}

 \medskip

 The proof of Theorem~\ref{thm:vecAp-char} for $T_\Qq$ requires two
 ancillary tools.  The first is a bilinear averaging operator that
 generalizes a linear operator introduced
 in~\cite{diening-harjulehto-hasto-ruzicka2010}.   Given $\pp\in \Pp$, 
define the $\pp$-average
\begin{equation*}
 \langle h \rangle_{\pp,Q}:= \frac{\|h\chi_Q \|_{\pp}}{\|\chi_Q \|_{\pp}},
\end{equation*}
and given a disjoint family of cubes $\mathcal{Q}$ define the $\pp$-averaging operator
\[ T_{\pp,\Qq}f(x) = \sum_{Q\in \Qq} \langle h \rangle_{\pp,Q}\cdot
  \chi_Q(x). \]
In \cite[Corollary~7.3.21]{diening-harjulehto-hasto-ruzicka2010} they
showed that if $\pp \in LH$, then
\begin{equation} \label{eqn:T-norm}
\|T_{\pp,\Qq}f \|_\pp \lesssim \|f\|_\pp.
\end{equation}

\smallskip

We define the bilinear
$\pp$-average operator analogously:  given $\pap,\,\pbp$
and a family of disjoint cubes $\Qq$, let 
    \begin{equation*}
      \vec{T}_{\vecpp,\mathcal{Q}}(f_1,f_2)(x)
= \sum_{Q \in \mathcal{Q}}\frac{\|f_1\chi_Q \|_{\pap}
\|g_1\chi_Q\|_{\pbp}}{\|\chi_Q \|_{\pp}}\cdot\chi_Q (x).
    \end{equation*}

\begin{lemma}\label{Bound-of-Tp}
Given $\pap,\,\pbp \in \Pp$, suppose $p_j(\cdot)\in LH$, $j=1,2$.
Then 
    \begin{equation*}
      \sup_\Qq\|\vec{T}_{\pp,\mathcal{Q}}(f_1,f_2)\|_{\pp}
\lesssim \|f_1\|_{\pap}\|f_2\|_{\pbp},
    \end{equation*}
where the supremum is taken over all collections $\Qq$ of disjoint
cubes.
\end{lemma}

\begin{proof}
Since $\pap,\,\pbp \in LH$,  $\pp \in LH$, and so by
Lemma~\ref{lemma:harmonic}, 
\[ \|\chi_Q\|_\pp \approx |Q|^{\frac{1}{p_Q}}
= |Q|^{\frac{1}{(p_1)_Q}}|Q|^{\frac{1}{(p_2)_Q}}
\approx \|\chi_Q\|_\pap \|\chi_Q\|_\pbp. \]
Therefore,
\[ \vec{T}_{\pp,\mathcal{Q}}(f_1,f_2)(x)
\approx \sum_{Q\in \mathcal{Q}} \langle f_1 \rangle_{\pap,Q}
\langle f_2 \rangle_{\pbp,Q}\cdot\chi_Q, \]
and so by Lemma~\ref{lemma:gen-holder} and~\eqref{eqn:T-norm},
  \begin{align*}
 \|\vec{T}_{\pp,\mathcal{Q}}(f_1,f_2)\|_{\pp}
& \lesssim \bigg\|\sum_{Q \in \mathcal{Q}}
\avpafa\avpbfb\cdot \chi_Q\bigg\|_{\pp} \\
  &    \lesssim \bigg\|\sum_{Q \in \mathcal{Q}}
\avpafa\cdot \chi_Q\sum_{Q \in \mathcal{Q}}
\avpbfb\cdot\chi_Q\bigg\|_{\pp} \\
& \lesssim   \|T_{\pap,\Qq}f_1\|_{\pap}\|T_{\pbp,\Qq}f_2\|_{\pbp} \\
& \lesssim \|f_1\|_{\pap}\|f_2\|_{\pbp}.
  \end{align*}
\end{proof}

The second tool is a summation property.  Given $\pap,\,\pbp\in \Pp$
suppose $\pp$ is 
such that $p_-\geq 1$.  Then
we say that $\vecpp \in \vec{G}$ if for every family of disjoint cubes
$\Qq$, 
\begin{equation*}
\sum_{Q \in \mathcal{Q}} \|f_1\chi_Q\|_{\pap}\|f_2\chi_Q\|_{\pbp}
\|h\chi_Q\|_{\cpp}
\lesssim \|f_1\|_{\pap}\|f_2\|_{\pbp}\|h\|_{\cpp},
\end{equation*}
where the implicit constant is independent of $\Qq$.

\begin{remark}
The linear version of property $G$ is due to Berezhno{\u\i}~\cite{MR1622773} in the setting
of Banach function spaces.  See
also~\cite{diening-harjulehto-hasto-ruzicka2010} where it is used to
prove~\eqref{eqn:T-norm}. 
\end{remark}

\begin{lemma}\label{Gcondition}
Given $\pap,\,\pbp\in \Pp$,  suppose $p_j(\cdot)\in LH$, $j=1,2$.  Then
$\vecpp \in \vec{G}$.
\end{lemma}

\begin{proof}
Since both $\pp,\,\cpp \in LH$,
by Lemma~\ref{lemma:harmonic}, for any cube $Q$
\[ \|\chi_Q\|_\pp \|\chi_Q\|_\cpp 
\approx |Q|^{\frac{1}{p_Q}}|Q|^{\frac{1}{(p')_Q}}=|Q|. \]
Hence, by Lemma~\ref{lemma:gen-holder},~\eqref{eqn:T-norm} and Lemma~\ref{Bound-of-Tp},
   \begin{align*} 
 &     \sum_{Q \in \mathcal{Q}} \| f_1 \chi_Q\|_{\pap} 
\|f_2\chi_Q \|_{\pbp} \| h \chi_Q\|_{\cpp}\\
&\qquad \qquad \approx \sum_{Q\in \mathcal{Q}} \int_\subRn 
\frac{\| f_1\chi_Q\|_{\pap} \| f_2\chi_Q\|_{\pbp}}{\|\chi_Q \|_{\pp}}
\frac{\| h\chi_Q\|_{\cpp}}{\|\chi_Q\|_{\cpp}}\cdot\chi_Q\,dx\\
&\qquad \qquad \lesssim \int_{\subRn}\sum_{Q\in \mathcal{Q}}
\frac{\| f_1\chi_Q\|_{\pap} \| f_2\chi_Q\|_{\pbp}}{\|\chi_Q \|_{\pp}}\cdot
  \chi_Q 
\sum_{Q\in \mathcal{Q}}
\frac{\| h\chi_Q\|_{\cpp}}{\|\chi_Q\|_{\cpp}}\cdot\chi_Q\,dx\\
 &\qquad \qquad =    
\int_{\subRn}\vec{T}_{\vecpp,\mathcal{Q}}(f_1,f_2)\;T_{\cpp,\mathcal{Q}}h\,dx\\
 &\qquad \qquad \lesssim 
\|\vec{T}_{\vecpp,\mathcal{Q}}(f_1,f_2)\|_{\pp}\|T_{\cpp,\mathcal{Q}}h\|_{\cpp}\\
 &\qquad \qquad \lesssim \|f_1\|_{\pap}\|f_2\|_{\pbp}\|h\|_{\cpp}.  
\end{align*}      
\end{proof}

\begin{proof}[Proof of Theorem~\ref{thm:vecAp-char} for $T_\Qq$]
 We first prove that the $\A_\vecpp$ condition is sufficient.  Since $|T_\Qq(f_1,f_2)(x)|\leq T_\Qq(|f_1|,|f_2|)(x)$, we may assume
  without loss of generality that $f_1,\,f_2$ are non-negative.  Because
  $(p_j)_-\geq 1$, $j=1,2$, we have $2p_-\geq1$, and so $2\pp\in \Pp$.
  Therefore, by Lemmas~\ref{lemma:rescale} and~\ref{lemma:dual}, there
  exists $hw^{-\frac{1}{2}}\in L^{(2p)'(\cdot)}$,
  $\|hw^{-\frac{1}{2}}\|_{(2p)'(\cdot)}\leq 1$, such that
\begin{align*}
\|T_\Qq(f_1,f_2)w\|_\pp^{\frac{1}{2}} 
& = \|T_\Qq(f_1,f_2)^{\frac{1}{2}}w^{\frac{1}{2}} \|_{2\pp} \\
& \approx \int_\subRn T_\Qq(f_1,f_2)^{\frac{1}{2}}
  w^{\frac{1}{2}}hw^{-\frac{1}{2}}\,dx \\
& \leq \sum_{Q\in \Qq} \avfa_Q^{\frac{1}{2}} \avfb_Q^{\frac{1}{2}}  
\int_Q hw^{\frac{1}{2}}w^{-\frac{1}{2}}\,dx \\
& = \sum_{Q\in \Qq} 
\|f_1^{\frac{1}{2}}w_1^{\frac{1}{2}} w_1^{-\frac{1}{2}} \chi_Q\|_2
\|f_2^{\frac{1}{2}}w_2^{\frac{1}{2}} w_2^{-\frac{1}{2}} \chi_Q\|_2
\|hw^{\frac{1}{2}}w^{-\frac{1}{2}}\chi_Q\|_1 |Q|^{-1}; \\
\intertext{by Lemmas~\ref{lemma:gen-holder} and \ref{lemma:holder}, }
& \leq \sum_{Q\in \Qq} \bigg[
\|f_1^{\frac{1}{2}}w_1^{\frac{1}{2}}\chi_Q\|_{2\pap}
  \|w_1^{-\frac{1}{2}} \chi_Q\|_{2p_1'(\cdot)} \\
& \quad \times \|f_2^{\frac{1}{2}}w_2^{\frac{1}{2}}\chi_Q\|_{2\pbp} 
  \|w_2^{-\frac{1}{2}} \chi_Q\|_{2p_2'(\cdot)}
\|hw^{-\frac{1}{2}}\chi_Q\|_{(2p)'(\cdot)}
  \|w^{\frac{1}{2}} \chi_Q\|_{2\pp}|Q|^{-1} \bigg] \\
& \leq \sum_{Q\in \Qq} \bigg[
\|f_1^{\frac{1}{2}}w_1^{\frac{1}{2}}\chi_Q\|_{2\pap}
\|f_2^{\frac{1}{2}}w_2^{\frac{1}{2}}\chi_Q\|_{2\pbp} 
\|hw^{-\frac{1}{2}}\chi_Q\|_{(2p)'(\cdot)} \\
& \quad \times 
\|w_1^{-\frac{1}{2}} \chi_Q\|_{2p_1'(\cdot)} 
\|w_2^{-\frac{1}{2}} \chi_Q\|_{2p_2'(\cdot)}
\|w_1^{\frac{1}{2}}\chi_Q\|_{2\pap}\|w_2^{\frac{1}{2}}\chi_Q\|_{2\pbp}|Q|^{-1}
  \bigg];
\intertext{by Proposition~\ref{AP-charachterization}, Lemma~\ref{Gcondition}
  applied to the exponents $2\pap,\,2\pbp$, and Lemma~\ref{lemma:rescale},}
& \lesssim \sum_{Q\in \Qq} 
\|f_1^{\frac{1}{2}}w_1^{\frac{1}{2}}\chi_Q\|_{2\pap}
\|f_2^{\frac{1}{2}}w_2^{\frac{1}{2}}\chi_Q\|_{2\pbp} 
\|hw^{-\frac{1}{2}}\chi_Q\|_{(2p)'(\cdot)} \\
& \lesssim \|f_1^{\frac{1}{2}}w_1^{\frac{1}{2}}\|_{2\pap}
\|g_1w_2^{\frac{1}{2}}\|_{2\pbp} 
\|hw^{-\frac{1}{2}}\|_{(2p)'(\cdot)} \\
& \leq \|f_1w_1\|_{\pap}^{\frac{1}{2}}
\|g_1w_2\|_{\pbp}^{\frac{1}{2}}.
\end{align*}
Since the implicit constants are independent of our choice of $\Qq$,
we conclude that $\vec{w}\in
\A_\vecpp$ implies \eqref{TQequation}.

The converse, that the $\A_\vecpp$ condition is necessary,  follows from the
corresponding implication for $A_Q$ proved above. 
\end{proof}

\section{Proof of Theorem~\ref{thm:main}}
\label{section:proof-main}

In this section we prove Theorem~\ref{thm:main}.    As before, given
weights $w_1$ and $w_2$ we define $w=w_1w_2$ and let
$\vec{w}=(w_1,w_2, w)$.  Given exponents $\pap,\,\pbp \in \Pp$, we
define $\pp$ by~\eqref{eqn:p-defn} and let $\vecpp=(\pap,\pbp,\pp)$.  

\smallskip

We first prove the
necessity of the $\A_\vecpp$ condition.   This is an immediate
consequence of Theorem~\ref{thm:vecAp-char}.  Given any cube $Q$, we
have that 
\[ |A_Q(f_1,f_2)(x)| \leq \M(f_1,f_2)(x). \]
Therefore, given weights $w_1$, $w_2$ such that
the boundedness condition \eqref{eqn:main1} holds, we immediately have
that~\eqref{AQequation} holds, and so $\vec{w}\in \A_\vecpp$.

\begin{remark}
The proof that the $\A_\vecpp$ condition is necessary does not
require us to assume that the exponents are log-H\"older continuous.
\end{remark}
\medskip

The proof of the sufficiency of the $\A_\vecpp$ condition is
considerably more complicated.   Fix $\pap,\,\pbp \in \Pp$ such that
$(p_j)_->1$ and $p_j(\cdot) \in LH$, $j=1,2$.  Let $w_1,\,w_2$ be such
that $\vec{w} \in \A_\vecpp$.  

We begin with a series of
reductions.  First, for $t\in \{0,1/3\}^n$, define 
\[ \D_t =  \{ 2^{-k}([0,1)^n +j+(-1)^k t) :  k \in \Z, j \in \Z^n \}. \]
Each $\D_t$ is a ``$1/3$'' translate of the standard dyadic grid, and
has exactly the same properties as $\D_0$ defined above.  (Note that
the two definitions agree when $t=0$.)   Define the dyadic bilinear maximal
operator
\[ \M^{\D_t}(f_1,f_2)(x) = \sup_{Q\in \D_t} 
\avgint_Q |f_1(y)|\,dy \avgint_Q |f_2(y)|\,dy \chi_Q(x).  \]
Then we have the following remarkable inequality:  there exists a
constant $C(n)$ such that
\[ \M(f_1,f_2)(x) \leq C(n)\sum_{t\in \{0,1/3\}^n} 
\M^{\D_t}(f_1,f_2)(x). \]
This was first proved in \cite{MR3302105}.  (For the  linear case, see
also~\cite{CruzUribe:2016ji}).

Therefore, to prove that inequality~\eqref{eqn:main1} holds, it
suffices to prove it with $\M$ replaced by $\M^{\D_t}$, and in fact it
suffices to prove it for $\M^d=\M^{\D_0}$, since the same proof holds
for any dyadic grid $\D_t$ with different constants, where the
difference only depends on $t$.  (Below we will describe where this
difference arises.)

Second, we may assume that $f,\,g$ are non-negative, bounded functions
with compact support.  It is clear from the definition of $\M^d$ that we
may take them non-negative.  To show the approximation, it suffices to
note that given $f_1,\,f_2$, there exists a sequence of non-negative, bounded
functions of compact support, $g_k,\,h_k$, that increase pointwise to $f$ and $g$ and such
that 
\[ \lim_{k\rightarrow \infty} \M^d(g_k,h_k)(x)=\M^d(f_1,f_2)(x).  \]
In the linear case this is proved
in~\cite[Lemma~3.30]{cruz-fiorenza-book} and the same proof (with the
obvious changes) works in the bilinear case.   The desired result then
follows by Lemma~\ref{lemma:fatou}.

Third, we restate the desired inequality in an equivalent fashion.
Given an exponent $\pp\in \Pp_0$ and a weight $v$, define $\Lp_{v}$ to
be the quasi-Banach function space with norm
    \begin{equation*}
    \|g\|_{\Lp_{v}}
    := \inf \bigg\{ \lambda>0:
    \int_{\subRn}\left(\frac{\abs{g(x)}}{\lambda}\right)^{p(x)}v(x)\,dx\leq
    1  \bigg\}.
    \end{equation*}
In other words, $\Lp_v$ is defined exactly as $\Lp$ with Lebesgue
measure replaced by the measure $v\,dx$.    This norm has many of the
same basic properties as the $\Lp$ norm.

Let $u(\cdot)=w(\cdot)^\pp$ and
$\sigma_l(\cdot)=w_l(\cdot)^{-p_l'(\cdot)}$, $l=1,2$;  then
\[   (\sigma_{l}(x)w_{l}(x))^{p_{l}(x)}
=(w_{l}(x)^{1-p_{l}^{\prime}(x)})^{p_{l}(x)}
 = w_{l}(x)^{-p_{l}^{\prime}(x)}
 = \sigma_{l}(x).
\]
Therefore,
\begin{equation*}
\|\M^d(f_1\sigma_{1},f_{2}\sigma_{2})\|_{\Lp_{u}}
=\|\M^d(f_{1}\sigma_{1},f_{2}\sigma_{2})w\|_{\Lp},
\end{equation*}      
and for $l=1,2$,
  \begin{equation*}
  \|f_{l}\|_{L_{\sigma_{l}}^{p_{l}(\cdot)}}
= \|f_{l}\sigma_{l}w_{l}\|_{L^{p_{l}(\cdot)}}.  
\end{equation*}     
 Hence,  it will suffice to prove that 
 \begin{equation*}\label{EQ-boundenessofM2}
 \|\M^d(f_{1}\sigma_{1},f_{2}\sigma_{2})\|_{\Lp_{u}}
\lesssim 
\|f_1\|_{\Lpa_{\sigma_{1}}}\|f_{2}\|_{\Lpb_{\sigma_{2}}},
      \end{equation*}
since if we replace $f_{l}$ by $f_{l}/\sigma_{l}$, $l=1,2$, we
get~\eqref{eqn:main1}.

Finally, by homogeneity  we may assume without loss of generality that
$\|f_l\|_{L^{p_l(\cdot)}_{\sigma_{l}}}=1$, $l=1,2$, which by
Lemma~\ref{lemma:mod-norm} (which holds in this setting) implies that
\[ \int_\subRn |f_l(x)|^{p_l(x)}\sigma_l(x)\,dx \leq 1. \]
Thus it will suffice to prove
that
 \begin{equation*}
 \|\M^{d}(f_{1}\sigma_{1},f_{2}\sigma_{2})\|_{\Lp_{u}}\lesssim 1,
        \end{equation*}
which, again by Lemma~\ref{lemma:mod-norm},  is
 equivalent to proving that
 \begin{equation}\label{modular_vertion_Maxboundeness}
      \int_{\subRn}\M^{d}(f_{1}\sigma_{1},f_{2}\sigma_{2})^{p(x)}u(x)\,dx
 \lesssim 1
      \end{equation}             
with a constant independent of $f_l$, $l=1,2$.      
    
\medskip

\smallskip

We now begin our main estimate, which is to prove
that~\eqref{modular_vertion_Maxboundeness} holds.
Define the functions
\begin{align*}
  h_{1}=f_{1}\chi_{\lbrace f_{1}>1\rbrace}, \;
& h_{2}=f_{1}\chi_{\lbrace f_{1}\leq 1 \rbrace}, \\ 
  h_{3}=f_{2}\chi_{\lbrace f_{2}>1\rbrace},\;
& h_{4}=f_{2}\chi_{\lbrace f_{2} \leq 1\rbrace},
\end{align*}
and for brevity define
\[     \rho(1)= 1, \quad \rho(2)= 1, \quad
   \rho(3)= 2, \quad \rho(4)= 2.  \]
Then we can write
  \begin{align*}
\int_\subRn \M^{d}\left( f_{1}\sigma_{1},
    f_{2}\sigma_{2}\right)(x)^{p(x)}u(x)\,dx
 &\leq 
\int_\subRn \M^{d}\left( h_{1}\sigma_{1}, h_{3}\sigma_{2}\right)(x)
   ^{p(x)}u(x)\,dx \\
& \qquad 
+\int_\subRn \M^{d}\left( h_{1}\sigma_{1}, h_{4}\sigma_{2}\right)(x)
   ^{p(x)}u(x)\,dx\\ 
&\qquad + 
\int_\subRn \M^{d}\left( h_{2}\sigma_{1},
  h_{3}\sigma_{2}\right)(x)^{p(x)}u(x)\,dx \\
& \qquad 
+ \int_\subRn \M^{d}\left( h_{2}\sigma_{1}, h_{4}\sigma_{2}\right)(x)
 ^{p(x)}u(x)\,dx\\
   &= I_{1}+I_{2}+I_{3}+I_{4}.
  \end{align*}
We will estimate each term on the right separately.  The integral $I_1$
is the ``local'' term and the estimate will use the $LH_0$ condition.
The integral $I_4$ is the ``global'' term and the estimate will use the
$LH_\infty$ condition.  The estimates  of $I_2$ and $I_3$ involve both
local and global estimates and are the most complicated:  this is where
our proof diverges most significantly from the linear case.  Note,
however, that the estimates for these integrals are the same (making
the obvious changes) so we will only estimate $I_2$. 
  
\subsection*{The estimate for $I_1$:}
We begin by forming the bilinear Calder\'on-Zygmund cubes associated
with $\M^d(h_1\sigma_1,h_3\sigma_2)$.  For the details of this decomposition,
see~\cite{MR2483720}.      Fix    $a> 2^{2n}$ and for each $k\in \Z$ define     
  \begin{equation*}
     \Omega_{k}
= \{ x\in \R^n: \M^{d}(h_{1}\sigma_{1},h_{3}\sigma_{2})(x)> a^{k} \}.
   \end{equation*}
   Then $\Omega_{k}= \bigcup_{j}{Q_{j}^{k}}$ where
   $\lbrace\Q\rbrace_{k,j}$ is a family of maximal dyadic cubes
   contained in $\Omega_k$ with the property that
\[ a^k < \langle h_1\sigma_1\rangle_\Q 
\langle h_3\sigma_2\rangle_\Q \leq a^{k+1}.  \] 
 Moreover, since $\Omega_{k+1}\subset \Omega_{k}$,
   the sets $E_{j}^{k}=\Q \setminus \Omega_{k+1}$ are pairwise
   disjoint and there exists $0<\alpha<1$ such that
       \begin{equation*}\label{EQ-Ejkestimate}
  \alpha |Q_j^k| < |E_j^k|.  
       \end{equation*}
 By Corollary~\ref{cor:Ainfty}, $u$ and $\sigma_l$,
$l=1,2$, are $A_\infty$ weights, so by
Lemma~\ref{lemma:Ainfty-prop} there exists $0<\beta<1$ such that 
\[ \beta u(\Q) \leq u(E_j^k), \qquad \beta \sigma_l(\Q) \leq
  \sigma_l(E_j^k). \]
We will use this fact repeatedly throughout the proof without further
comment.

\smallskip  

We can now estimate $I_1$ as follows:   
 \begin{align*}
I_1 
& =
  \int_{\subRn}\M^{d}(h_{1}\sigma_{1},h_{3}\sigma_{2})(x)^{p(x)}u(x)\,dx \\
&\leq\sum\limits_{k=0}^{\infty}
\int_{\Omega_{k}\setminus\Omega_{k+1}}a^{(k+1)p(x)}u(x)\,dx\\ 
     &\lesssim \sum_{k,j}
\int_{E_{j}^{k}}\prod_{l=1,3}\langle
       h_{l}\sigma_{\rho(l)}\rangle_{\Q}^{p(x)}u(x)\,dx \\
     &= \sum_{k,j}\int_{E_{j}^{k}}\prod_{l=1,3}
\bigg(\int_{\Q}{h_{l}\sigma_{\rho(l)}\,dy}\bigg)^{p(x)}
|\Q|^{-2p(x)}u(x)\,dx.
\end{align*}   
   
 Since $h_1(x)\geqslant 1$ or $h_1(x)=0$, we have that
 \begin{equation}\label{Bound_function_geq1}
 \int_{\Q}h_1(y)\sigma_{1}(y)\,dy
\leq   \int_{\Q}h_{1}(y)^{p_{1}(y)}\sigma_{1}(y)\,dy
\leq \int_\subRn f_1(y)^{p_1(y)}\sigma_1(y)\,dy \leq 1. 
\end{equation}
The same estimate holds for $h_3$.   For each $j,k$ define
 \begin{equation*}
 \frac{1}{q(\Q)}=\frac{1}{(p_{1})_{-}(\Q)}+\frac{1}{(p_{2})_{-}(\Q)},
 \end{equation*}
 and note that  for $x\in \Q$, $q(\Q)\leq p_{-}(\Q)\leq p(x)$. Thus,
   \begin{align*}
     I_{1} 
&\leq  \sum_{k,j}\int_{E_{j}^{k}}
\prod_{l=1,3}\bigg(\int_{\Q}h_{l}(y)\sigma_{\rho(l)}(y)\,dy\bigg)
^{q(\Q)}|\Q|^{-2p(x)}u(x)\,dx\\
& \leq \sum_{k,j}\int_{E_{j}^{k}}
\prod_{l=1,3}\bigg(\frac{1}{\sigma_{\rho(l)}(\Q)}
\int_{\Q}h_{l}(y)^{\frac{p_{l}(y)}{(p_{l})_{-}(\Q)}}\sigma_{\rho(l)}(y)\,dy
\bigg)^{q(\Q)} \\
& \qquad \qquad \times 
\sigma_{\rho(l)}(\Q)^{q(\Q)}|\Q|^{-2p(x)}u(x)\,dx.
   \end{align*}
   
 By Hölder's inequality with measure $\sigma_{l}\,dx$, for $l=1,3$,
  \begin{multline}\label{EQ-estimateHölder}
 \bigg(\frac{1}{\sigma_{\rho(l)}(\Q)}
\int_{\Q}h_{l}(y)^{\frac{p_{l}(y)}{(p_{l})_{-}(\Q)}}
\sigma_{\rho(l)}(y)\,dy\bigg)^{q(\Q)} \\
 \leq \bigg(\frac{1}{\sigma_{\rho(l)}(\Q)}
\int_{\Q}h_{l}(y)^{\frac{p_{l}(y)}{(p_{l})_{-}}}\sigma_{\rho(l)}(y)\,dy
\bigg)^{(p_{l})_{-}\frac{q(\Q)}{(p_{l})_{-}(\Q)}}
= \langle h_{l}^{\frac{p_{l}(\cdot)}{(p_{l})_{-}}}
 \rangle_{\sigma_{\rho(l)},Q}^{(p_{l})_{-}\frac{q(\Q)}{(p_{l})_{-}(\Q)}}.
\end{multline}
Further,  we claim  that
   \begin{equation}\label{EQ-I1estimte}
\int_{E_{j}^{k}}\prod_{l=1,3}\sigma_{\rho(l)}(\Q)^{q(\Q)}
|\Q|^{-2p(x)}u(x)\,dx
\lesssim \sigma_{1}(\Q)^{\frac{q(\Q)}{(p_{1})_{-}(\Q)}}
\sigma_{2}(\Q)^{\frac{q(\Q)}{(p_{2})_{-}(\Q)}}.
   \end{equation}

If we assume this for the moment, then we can argue as follows:
since 
 \begin{equation*}
 1=\frac{q(\Q)}{(p_{1})_{-}(\Q)}+\frac{q(\Q)}{(p_{2})_{-}(\Q)},
 \end{equation*}
by \eqref{EQ-estimateHölder} and Young's inequality,
 \begin{align}
     I_1 
\nonumber &\lesssim \sum_{k,j}\prod_{l=1,3}
\langle h_{l}^{\frac{p_{l}(\cdot)}{(p_{l})_{-}}}
 \rangle_{\sigma_{\rho(l)},Q}^{(p_{l})_{-}\frac{q(\Q)}{(p_{l})_{-}(\Q)}}
\sigma_{\rho(l)}(\Q)^{\frac{q(\Q)}{(p_{l})_{-}(\Q)}}\\
\label{eqn:final-I1-est}     &\leq  \sum_{k,j}\sum_{l=1,3}
\langle h_{l}^{\frac{p_{l}(\cdot)}{(p_{l})_{-}}}
 \rangle_{\sigma_{\rho(l)},Q}^{(p_{l})_{-}}\sigma_{\rho(l)}( Q_{j}^{k})\\
\nonumber  &\lesssim  \sum_{k,j}\sum_{l=1,3}
\langle h_{l}^{\frac{p_{l}(\cdot)}{(p_{l})_{-}}}
 \rangle_{\sigma_{\rho(l)},Q}^{(p_{l})_{-}}\sigma_{\rho(l)}(
             E_{j}^{k}).\\
\intertext{ By 
Lemma~\ref{lemma:wtd-max-bound}, since $(p_l)_->1$,}
\nonumber     &\leq \sum_{l=1,3} \int_{\subRn}
M_{\sigma_{\rho(l)}}^{d}(h_{l}^{\frac{p_{l}(\cdot)}{(p_{l})_{-}}})(x)
^{(p_{l})_{-}}\sigma_{\rho(l)}(x)\,dx\\
\nonumber     & \lesssim \sum_{l=1,3} \int_{\subRn}
h_{l}(x)^{p_{l}(x)}\sigma_{\rho(l)}(x)\,dx\\
\nonumber     & \lesssim 1.
     \end{align}
     
 Therefore, to complete the estimate of $I_1$ we will
 prove~\eqref{EQ-I1estimte}.   First, rewrite the left-hand side as
 follows: 
   \begin{align*}
    & \int_{E_{j}^{k}}\prod_{l=1,3}\sigma_{\rho(l)}(\Q)^{q(\Q)}
|\Q|^{-2p(x)}u(x)\,dx \\
    & \qquad \leq  \prod_{l=1,3}
\bigg(\frac{\sigma_{\rho(l)}(\Q)}
{\|w_{\rho(l)}^{-1}\chi_\Q\|_{p_{l}^{\prime}(\cdot)}}
      \bigg)^{q(\Q)}\\
    & \qquad \qquad \times
\int_{\Q}\bigg(\prod_{l=1,3}
\|w_{\rho(l)}^{-1}\chi_\Q\|_{p_{l}^{\prime}(\cdot)}^{q(\Q)-p(x)}\bigg)
\bigg(\prod_{l=1,3}
\|w_{\rho(l)}^{-1}\chi_{\Q}\|_{p_{l}^{\prime}(\cdot)}^{p(x)}
|\Q|^{-2p(x)}u(x)\bigg)\,dx.
\end{align*}      
By the $\mathcal{A}_{\vec{p}(\cdot)}$ condition we  have that there is
a constant $c$ such that
 \begin{equation*}
\big\|c|\Q|^{-2}\prod_{l=1 }^{2}
\|w_{l}^{-1}\chi_{\Q}\|_{p_{l}^{\prime}(\cdot)}w\chi_{\Q}
\big\|_{p(\cdot)}
\leq 1,
       \end{equation*}
 which by Lemma~\ref{lemma:mod-norm} implies that
 \begin{equation} \label{eqn:modular-Ap}
 \int_{\Q}\prod_{l=1 }^{2}
\|w_{l}^{-1}\chi_{\Q}\|_{p_{l}^{\prime}(\cdot)}^{p(x)}
|\Q|^{-2p(x)}u(x)\,dx \lesssim 1.
     \end{equation}
Hence, to prove \eqref{EQ-I1estimte} it will suffice to show that for
$l=1,\,2$, 
 \begin{equation}\label{EQ-I11estimate}
\bigg(\frac{\sigma_{l}(\Q)}
{\|w_{l}^{-1}\chi_{\Q}\|_{p_{l}^{\prime}(\cdot)}}\bigg)^{q(\Q)}
\lesssim  \sigma_{l}(\Q)^{\frac{q(\Q)}{(p_{l})_{-}(\Q)}}
\end{equation}        
and
 \begin{equation}\label{EQ-I12estimate}
 \|w_{l}^{-1}\chi_{\Q}\|_{p_{l}^{\prime}(\cdot)}^{q(\Q)-p(x)}
\lesssim 1.
   \end{equation}

   We first prove~\eqref{EQ-I11estimate}.  Suppose that
   $\|w_{l}^{-1}\chi_{\Q}\|_{p_{l}^{\prime}(\cdot)} >1$. Then by
   Lemma~\ref{lemma:mod-norm}, since $(p_l^{\prime})_{\pm}(\Q)=(p_{l})_{\mp}(\Q)^{\prime}$, we have
   that
\begin{equation*}
\bigg(\frac{\sigma_{l}(\Q)}
{\| w_{l}^{-1}\chi_\Q\|_{p_{l}^{\prime}(\cdot)}}\bigg)^{q(\Q)}
\leq \bigg( \sigma_{l}(\Q)
^{1-\frac{1}{(p_{l})_{-}(\Q)^{\prime}}} 
\bigg)^{q(\Q)}
=\sigma_{l}(\Q)^{\frac{q(\Q)}{(p_{l})_{-}(\Q)}}.
\end{equation*}
On the other hand, if $\|w_{l}^{-1}\chi_{\Q}\|_{p_{l}^{\prime}(\cdot)}
\leq 1$, 
 \begin{multline*}
 \frac{\sigma_l(\Q)}{\|w_l^{-1}\chi_\Q\|_{p_l'(\cdot)}}
 \leq  \sigma_l(\Q)^{1-\frac{1}{(p_l)_+(\Q)^{'}}}  
 = \sigma_l(\Q)^{\frac{1}{(p_l)_+(\Q)}}  \\
 = \sigma_l(\Q)^{\frac{1}{(p_l)_-(\Q)}}
 \sigma_l(\Q)^{\frac{1}{(p_l)_+(\Q)}-\frac{1}{(p_l)_-(\Q)}}. 
 \end{multline*}
 Again by Lemma~\ref{lemma:mod-norm}, and then by
 Lemma~\ref{lemma:rescale} and Lemma~\ref{lemma:lemma3.3},
 \begin{align*}
 \sigma_l(\Q)^{\frac{1}{(p_l)_+(\Q)}-\frac{1}{(p_l)_-(\Q)}}
 &\leq  \|w_l^{-\frac{1}{2}}\chi_\Q\|_{2p_l'(\cdot)}
 ^{[2(p_l')_{-}]\big(\frac{1}{(p_l)_+(\Q)}-\frac{1}{(p_l)_-(\Q)}\big)}\\
 &=\|w_l^{-\frac{1}{2}}\chi_\Q\|_{2p_l'(\cdot)}
 ^{[2(p_l')_{-}]\big(1-\frac{1}{(p_l)_+(\Q)^{'}}-1+\frac{1}{(p_l)_-(\Q)^{'}}\big)} \\
 &= \|w_l^{-\frac{1}{2}}\chi_\Q\|_{2p_l'(\cdot)}
 ^{[2(p_l')_{-}]\big(\frac{1}{(p_l^{\prime})_+(\Q)}-\frac{1}{(p_1^{\prime})_-(\Q)}\big)} \\
 & \leqslant \|w_l^{-\frac{1}{2}}\chi_\Q\|_{2p_l'(\cdot)}
 ^{c[(2p_l^{\prime})_-(\Q)-(2p_l^{\prime})_+(\Q)]}\\
 & \lesssim 1.
 \end{align*}
 Hence, 
   \begin{equation}\label{eq:sigma-w-estimate}
  \bigg(\frac{\sigma_1(\Q)}{\|w_1^{-1}\chi_\Q\|_{p_1'(\cdot)}}\bigg)^{q(\Q)}
 \lesssim \sigma_1(\Q)^{\frac{q(\Q)}{(p_1)_-(\Q)}}.
 \end{equation}

We now prove~\eqref{EQ-I12estimate}.  If
$\|w_{l}^{-1}\chi_{\Q}\|_{p_{l}^{\prime}(\cdot)}  
\geq 1$, then this is immediate.  If
$\|w_{l}^{-1}\chi_{\Q}\|_{p_{l}^{\prime}(\cdot)} <1$, then by
Lemma~\ref{lemma:rescale}, and then by
Propositions~\ref{AP-charachterization}
and~\ref{q-relation}, 
\[
  \|w_{l}^{-1}\chi_{\Q}\|_{p_{l}^{\prime}(\cdot)}^{q(\Q)-p(x)}
=  \|w_{l}^{-\frac{1}{2}}\chi_{\Q}\|_{2p_{l}^{\prime}(\cdot)}^{2q(\Q)-2p(x)}
\lesssim 1.
\]
This completes the estimate of $I_1$.                            

\medskip    
     
 \subsection*{The estimate for $I_2$:}
We first form the bilinear Calder\'on-Zygmund cubes associated with
$\M^d(h_{1}\sigma_{1},h_{4}\sigma_{2})$ and we use the same notation
as we did in the estimate for $I_1$.   To estimate this term $I_2$ we need
to divide the cubes $\Q$ into three sets:  small cubes close to
the origin, large cubes close to the origin, and cubes (of all sizes) far
from the origin.  To make this precise, let $\{P_i\}_{i=1}^{2^n}$ be
the $2^n$ dyadic cubes adjacent to the origin, $|P_i|\geq 1$,  that are so
large that  if
$Q$ is any dyadic cube equal to or adjacent to $P_i$ in the same quadrant, and $|P_i|=|Q|$, then,
$u(Q)\geq 1$ and $\sigma_l(Q)\geq 1$, $l=1,2$.  The existence of such
cubes follows from Lemma~\ref{lemma:Ainfty-prop} and
Corollary~\ref{cor:Ainfty}.    Let $P=\bigcup_i P_i$.   We can then partition the cubes $\{\Q\}$ into three disjoint sets:
\begin{align*}
\mathscr{F} & =\lbrace (k,j): \Q \subset P_i \text{ for some } i\rbrace,\\
\mathscr{G}& =\lbrace (k,j): P_i \subset \Q \text{ for some } i \rbrace,\\
\mathscr{H} & =\lbrace (k,j): \Q \cap P_i = \emptyset \text{ for all } i
              \rbrace. 
\end{align*}

We now estimate $I_2$, arguing as we
did at the beginning of the estimate for $I_1$:
  \begin{multline*}
 \int_{\subRn}\M^{d}(h_{1}\sigma_{1},h_{4}\sigma_{2})(x)
^{p(x)}u(x)\,dx
\lesssim
\sum_{k,j}\int_{E_{j}^{k}}\prod_{l=1,4}
\langle h_{l}\sigma_{\rho(l)}\rangle_{\Q}^{p(x)}u(x)\,dx \\
= \sum_{(k,j)\in\mathscr{F}}
+\sum_{(k,j)\in\mathscr{G}}
+\sum_{(k,j)\in\mathscr{H}}
= J_{1}+J_{2}+J_{3}.
\end{multline*}
We will estimate each sum in turn.  

\begin{remark}
Throughout the rest of this
proof, we will allow the implicit constants to depend on
$\sigma_l(P)$ or $u(P)$.  The choice of the $P_i$ is the one place
where the proof depends on the fact that we are working with the
dyadic grid $\D_0$.  For the grids $\D_t$ we will replace the origin
by its translate $\pm t$, where the sign will depend on the scale at
which we choose the $P_i$.   See Remark~\ref{remark:grids} below for
where the dyadic grid and the choice of the $P_i$ affects the proof.
\end{remark}

\subsubsection*{The estimate for $J_{1}$:}
Since $h_{4}\leq 1$ and $p_+<\infty$,  we have that
\begin{align*}
  J_{1}
  & = \sum_{(k,j)\in \mathscr{F}}\int_{E_{j}^{k}}
    \prod_{l=1,4}\langle h_{l}\sigma_{\rho(l)}\rangle_{\Q}^{p(x)}u(x)\,dx
  \\
  & \leq \sum_{(k,j)\in\mathscr{F}}\int_{E_{j}^{k}}
    \langle h_{1}\sigma_{1}\rangle_{\Q}^{p(x)}\langle
    \sigma_{2}\rangle_{\Q}^{p(x)}u(x)\,dx \\
  & = \sum_{(k,j)\in\mathscr{F}}\int_{E_{j}^{k}}
    \bigg(\int_{\Q} h_1\sigma_1\,dy\bigg)^{p(x)}
    \sigma_{2}(\Q)^{p(x)-q(\Q)}
    \sigma_{2}(\Q)^{q(\Q)}|\Q|^{-2p(x)}u(x)\,dx; \\
  \intertext{by inequalities~\eqref{Bound_function_geq1} and~\eqref{EQ-estimateHölder},}
  & \leq \sum_{(k,j)\in\mathscr{F}}\int_{E_{j}^{k}}
    \bigg(\int_{\Q} h_1\sigma_1\,dy\bigg)^{q(\Q)}
    \sigma_{2}(\Q)^{p(x)-q(\Q)}
    \sigma_{2}(\Q)^{q(\Q)}|\Q|^{-2p(x)}u(x)\,dx \\
  & = \sum_{(k,j)\in\mathscr{F}}\int_{E_{j}^{k}}
    \langle h_{1}\rangle_{\sigma_{1},\Q}^{q(\Q)}
    \sigma_{2}(\Q)^{p(x)-q(\Q)}\sigma_{1}(\Q)^{q(\Q)}\sigma_{2}(\Q)^{q(\Q)}
    |\Q|^{-2p(x)}u(x)\,dx \\
  &\leq
    \sum_{(k,j)\in\mathscr{F}}\big(\sigma_{2}(\Q)+1\big)^{p_{+}(\Q)-q(\Q)}
    \langle h_{1}^{\frac{p_{1}(\cdot)}{(p_{1})_{-}}}\rangle_{\sigma_{1},\Q}
    ^{(p_{1})_{-}\frac{q(\Q)}{(p_{1})_{-}(\Q)}}\\
  &   \qquad \qquad\times
    \int_{E_{j}^{k}}\sigma_{1}(\Q)^{q(\Q)}\sigma_{2}(\Q)^{q(\Q)}|\Q|^{-2p(x)}u(x)\,dx.
  \\
  \intertext{Let $q_-$ be defined as in~\eqref{eqn:q-minus}.   By~\eqref{EQ-I1estimte} we can
  estimate the integral:}
  &\lesssim
    \big(\sigma_{2}(P)+1\big)^{p_{+}-q_-}\sum_{(k,j)\in\mathscr{F}}
    \langle h_{1}^{\frac{p_{1}(\cdot)}{(p_{1})_{-}}}
    \rangle_{\sigma_{1},\Q}^{(p_{1})_{-}\frac{q(\Q)}{(p_{1})_{-}(\Q)}} 
    \sigma_{1}(\Q)^{\frac{q(\Q)}{(p_{1})_{-}(\Q)}}\sigma_{2}(\Q)^{\frac{q(\Q)}{(p_{2})_{-}(\Q)}}.\\
  \intertext{Therefore, by Young's inequality
  and by
  Lemma~\ref{lemma:wtd-max-bound}, }
  &\leq \big(\sigma_{2}(P)+1\big)^{p_{+}-q_{-}}\sum_{(k,j)\in
    \mathscr{F}}\big[
    \langle h_{1}^{\frac{p_{1}(\cdot)}{(p_{1})_{-}}}\rangle_{\sigma_{1},\Q}^{(p_{1})_{-}}
    \sigma_{1}(\Q)+\sigma_{2}(\Q)\big]  \\
  &\lesssim    \big(\sigma_{2}(P)+1\big)^{p_{+}-q_{-}}\sum_{(k,j)\in
    \mathscr{F}}\big[
    \langle
    h_{1}^{\frac{p_{1}(\cdot)}{(p_{1})_{-}}}\rangle_{\sigma_{1},\Q}^{(p_{1})_{-}}
    \sigma_{1}(E_{j}^{k})+\sigma_{2}(E_{j}^{k})\big]\\
  & \lesssim \sum_{(k,j)\in \mathscr{F}}\int_{E_{j}^{k}}
    \mathcal{M}_{\sigma_{1}}^{d}\big(f_{1}^{\frac{p_{1}(\cdot)}{(p_{1})_{-}}}\big)(x)
    ^{(p_{1})_{-}}\sigma_{1}(x)\,dx
    + \sum_{(k,j)\in \mathscr{F}} \sigma_{2}(E_{j}^{k})\\
  & \lesssim \int_{\subRn} f_{1}(x)^{p_{1}(x)}\sigma_{1}(x)\,dx
    + \sigma_{2}(P)\\
  & \lesssim 1.
\end{align*}

\medskip

\subsubsection*{The estimate for $J_{2}$:} 
Given $(k,j) \in \mathscr{G}$, since $P_i\subset \Q$, we have that
$1\leq \sigma_2(P_i)\leq \sigma_2(\Q)$.  Therefore, by
Lemma~\ref{lemma:p-infty-cond} applied twice to $w_2^{-\frac{1}{2}}\in
\A_{2p_2'(\cdot)}$, we get
\begin{multline*}
\frac{1}{|\Q|} 
\leq \frac{|P_{i}|}{|\Q|}
\lesssim \bigg(\frac{\sigma_{2}(P_{i})}
{\sigma_{2}(\Q)}\bigg)^{\frac{1}{2(p_{2}^{\prime})_{\infty}}}
\lesssim \sigma_{2}(\Q)^{-\frac{1}{2(p_{2}^{\prime})_{\infty}}}\\
 \lesssim\|w_{2}^{-\frac{1}{2}}\chi_{\Q}\|_{2\cpbp}^{-1}
        \lesssim \|w_{2}^{-1}\chi_{\Q}\|_{\cpbp}^{-\frac{1}{2}}.
     \end{multline*}
Hence, by Lemma~\ref{lemma:gen-holder},
\begin{align*}
\frac{1}{|\Q|^{2}}\int_{\Q} h_{4}(y)\sigma_{2}(y)\,dy
& \lesssim \|w_{2}^{-1}\chi_{\Q}\|_{\cpbp}^{-1}
\int_{\Q} h_{4}(y)\sigma_{2}(y)^{\frac{1}{p_2(y)}}\sigma_{2}(y)^{\frac{1}{p_2'(y)}}\,dy\\
& \lesssim \|w_{2}^{-1}\chi_{\Q}\|_{\cpbp}^{-1}
\|h_{4}\|_{L_{\sigma_{2}}^{p_{2}(\cdot)}}
\|\chi_{\Q}\|_{L_{\sigma_{2}}^{p_{2}^{\prime}(\cdot)}}\\
& \leq \|w_{2}^{-1}\chi_{\Q}\|_{\cpbp}^{-1}
\|f_{2}\|_{L_{\sigma_{2}}^{p_{2}(\cdot)}}
\|w_{2}^{-1}\chi_{\Q}\|_{\cpbp}\\
& \leq c_0.
\end{align*}

We can now estimate $J_2$.  By inequality~\eqref{Bound_function_geq1}
and Lemmas~\ref{lemma:p-infty-px} and~\ref{lemma:infty-bound}, there
exists $t>1$ such that
 \begin{align*}
J_{2} 
& =
\sum_{(k,j)\in \mathscr{G}}\int_{E_{j}^{k}}
c_0^{p(x)} \bigg(\int_{\Q}{h_{1}\sigma_{1}\,dy}\bigg)^{p(x)}
\bigg(\frac{c_0^{-1}}{|\Q|^{2}}
\int_{\Q}h_{4}\sigma_{2}\,dy\bigg)^{p(x)}u(x)\,dx\\
& \lesssim \sum_{(k,j)\in\mathscr{G}} c_0^{p_+}\int_{E_{j}^{k}}
{\bigg(\int_{\Q}h_{1}\sigma_{1}\,dy\bigg)^{p_{\infty}}
\bigg(\frac{c_0^{-1}}{|\Q|^{2}}
\int_{\Q}h_{4}\sigma_{2}\,dy}\bigg)^{p_{\infty}}u(x)\,dx\\
& \qquad \qquad \qquad 
+ \sum_{(k,j)\in \mathscr{G}}\int_{E_{j}^{k}}
\frac{u(x)}{(e+\abs{x})^{ntp_{-}}}\,dx \\
&\lesssim \sum_{(k,j)\in \mathscr{G}}
\prod_{l=1,4}\langle h_{l}\rangle_{\sigma_{\rho(l),\Q}}^{p_{\infty}}
\sigma_{1}(\Q)^{p_{\infty}}
\sigma_{2}(\Q)^{p_{\infty}}|\Q|^{-2p_{\infty}}u(E_{j}^{k})+1.
\end{align*}

We estimate each term in the product separately.  First, we claim that
\begin{equation}\label{P_infty_Bounded}
\sigma_{1}(\Q)^{p_{\infty}}\sigma_{2}(\Q)^{p_{\infty}}
|\Q|^{-2p_{\infty}}u(E_{j}^{k})
\lesssim \sigma_{1}(\Q)^{\frac{p_{\infty}}{(p_{1})_{\infty}}}
\sigma_{2}(\Q)^{\frac{p_{\infty}}{(p_{2})_{\infty}}}.
\end{equation}
Since $\sigma_l(\Q),\, u(\Q) \geq 1$, by
Lemma~\ref{lemma:p-infty-cond} (applied several times) 
and the definition of $\mathcal{A}_{\vec{p}(\cdot)}$, we have that
\begin{align*}  
\bigg[\sigma_{1}(\Q)\sigma_{2}(\Q)\bigg]^{p_{\infty}}
& \lesssim 
\bigg( \|w_{1}^{-\frac{1}{2}}\chi_{\Q}\|_{2\cpap}^{2(p_{1}^{\prime})_{\infty}}
\|w_{2}^{-\frac{1}{2}}\chi_{\Q}\|_{2\cpbp}^{2(p_{2}^{\prime})_{\infty}}\bigg)^{p_{\infty}} \\
& = \bigg(\|w_{1}^{-1}\chi_{\Q}\|_{\cpap}^{(p_{1}^{\prime})_{\infty}-1}
\|w_{2}^{-1}\chi_{\Q}\|_{\cpbp}^{(p_{2}^{\prime})_{\infty}-1}\bigg)^{p_{\infty}}
\bigg(\prod_{l=1}^{2} \|w_{l}^{-1}\chi_{\Q}\|_{p_{l}^{\prime}(\cdot)}
\bigg)^{p_{\infty}}\\
&\lesssim \bigg( \|w_{1}^{-1}\chi_{\Q}\|_{\cpap}^{(p_{1}^{\prime})_{\infty}-1}
\|w_{2}^{-1}\chi_{\Q}\|_{\cpbp}^{(p_{2}^{\prime})_{\infty}-1}\bigg)^{p_{\infty}}
\bigg(\frac{|\Q|^{2}}{\|w\chi_{\Q}\|_{\pp}}\bigg)^{p_{\infty}}\\
& \lesssim\bigg(
\sigma_{1}(\Q)^{\frac{(p_{1}^{\prime})_{\infty}-1}{(p_{1}^{\prime})_{\infty}}}
\sigma_{2}(\Q)^{\frac{(p_{2}^{\prime})_{\infty}-1}{(p_{2}^{\prime})_{\infty}}}
\bigg)^{p_{\infty}}\frac{|\Q|^{2p_{\infty}}}{u(\Q)}\\
&\leq\sigma_{1}(\Q)^{\frac{p_{\infty}}{(p_{1})_{\infty}}}
\sigma_{2}(\Q)^{\frac{p_{\infty}}{(p_{2})_{\infty}}}
\frac{|\Q|^{2p_{\infty}}}{u(E_{j}^{k})}.
\end{align*}
This proves \eqref{P_infty_Bounded}. 

Second, by Lemma~\ref{lemma:gen-holder} and again by
Lemma~\ref{lemma:p-infty-cond} we have that
  \begin{multline} \label{sigma1_Averange_Bound}
  \frac{1}{\sigma_{1}(\Q)}\int_{\Q}{h_{1}(y)\sigma_{1}(y)\,dy}
\lesssim 
\sigma_{1}(\Q)^{-1}\|h_{1}\|_{L_{\sigma_{1}}^{p_{1}(\cdot)}}
\|\chi_{\Q}\|_{L_{\sigma_{1}}^{p_{1}^{\prime}(\cdot)}} \\
\leq
\sigma_{1}(\Q)^{-1}\|f\|_{L_{\sigma_{1}}^{p_{1}(\cdot)}} 
\|w_{1}^{-1}\chi_{\Q}\|_{\cpap}
\lesssim  \sigma_{1}(\Q)^{\frac{1}{(p_{1}^{\prime})_{\infty}}-1}
\leq \sigma_{1}(P_i)^{-\frac{1}{(p_{1})_{\infty}}}
\lesssim 1.
\end{multline}

We can now continue the estimate of $J_2$.   Since
\begin{equation*}
1=\frac{p_{\infty}}{(p_{1})_{\infty}}+\frac{p_{\infty}}{(p_{2})_{\infty}},
\end{equation*}
by~\eqref{P_infty_Bounded} and Young's inequality,
\begin{align}
\nonumber  J_{2}
& \lesssim \sum_{(k,j)\in \mathscr{G}}
\prod_{l=1,4}\langle h_{l}\rangle_{\sigma_{\rho(l)},\Q}^{p_{\infty}}
\sigma_{\rho(l)}(\Q)^{\frac{p_{\infty}}{(p_{l})_{\infty}}} + 1\\
\label{eqn:J2-final-est} & \lesssim \sum_{(k,j)\in \mathscr{G}}
\langle h_{1}\rangle_{\sigma_{1},\Q}^{(p_{1})_{\infty}}
\sigma_{1}(\Q)
+ \sum_{(k,j)\in \mathscr{G}}
\langle h_{4}\rangle_{\sigma_{2},\Q}^{(p_{2})_{\infty}}\sigma_{2}(\Q)
  +1;\\
\nonumber  & \lesssim \sum_{(k,j)\in \mathscr{G}}
\langle c_0^{-1}h_{1}\rangle_{\sigma_{1},\Q}^{(p_{1})_{\infty}}
\sigma_{1}(E_{j}^{k})
+ \sum_{(k,j)\in \mathscr{G}}
\langle h_{4}\rangle_{\sigma_{2},\Q}^{(p_{2})_{\infty}}
\sigma_{2}(E_{j}^{k}) +1; \\
\intertext{by Lemmas~\ref{lemma:p-infty-px}
  and~\ref{lemma:infty-bound} there exists $t>1$ such that,}
\nonumber  &\lesssim \sum_{(k,j)\in \mathscr{G}}\int_{E_{j}^{k}}
\langle c_0^{-1}h_{1}\rangle_{\sigma_{1},\Q}^{p_{1}(x)}
\sigma_{1}(x)\,dx
+ \sum_{(k,j)\in \mathscr{G}}\int_{E_{j}^{k}}
\frac{\sigma_{1}(x)}{(e+|x|)^{tn(p_{1})_{-}}}\,dx\\ 
\nonumber    & \qquad \qquad \qquad 
 +\sum_{(k,j)\in \mathscr{G}}
\int_{E_{j}^{k}}{M}_{\sigma_{2}}^{d}h_{4}(x)^{(p_{2})_{\infty}}
\sigma_{2}(x)\,dx +1\\
\nonumber  & \lesssim \sum_{(k,j)\in \mathscr{G}}\int_{E_{j}^{k}}
\langle
   c_0^{-1}h_{1}^{\frac{p_{1}(\cdot)}{(p_{1})_{-}}}\rangle_{\sigma_{1},\Q}^{(p_{1})_{-}}
   \sigma_{1}(x)\,dx
+ \int_{\subRn}{M}_{\sigma_{2}}^{d}h_{4}(x)^{(p_{2})_{\infty}}
\sigma_{2}(x)\,dx + 1;\\
\intertext{by Lemma~\ref{lemma:wtd-max-bound} applied twice,} 
\nonumber  & \lesssim \int_{\subRn}
{M}_{\sigma_{1}}^{d}(h_{1}^{\frac{p_{1}(\cdot)}{(p_{1})_{-}}})(x)^{(p_{1})_{-}}
\sigma_{1}(x)\,dx
+ \int_{\subRn}h_{4}(x)^{(p_{2})_{\infty}}\sigma_{2}(x)\,dx + 1\\
\nonumber &\lesssim \int_{\subRn} h_{1}(x)^{p_{1}(x)}\sigma_{1}(x)\,dx
+ \int_{\subRn}h_{4}(x)^{(p_{2})_{\infty}}\sigma_{2}(x)\,dx + 1\\
\nonumber  & \lesssim \int_{\subRn}h_{4}(x)^{(p_{2})_{\infty}}\sigma_{2}(x)\,dx
   + 1.   \\         
\intertext{Finally, we again apply Lemmas~\ref{lemma:p-infty-px} and 
\ref{lemma:infty-bound} to get } 
\nonumber & \lesssim \int_{\subRn}h_{4}(x)^{p_{2}(x)}\sigma_{2}(x)\,dx
+\int_{\subRn}\frac{\sigma_{2}(x)}{(e+|x|)^{tn(p_{2})_{-}}}\,dx +1\\
\nonumber & \lesssim 1,
    \end{align}
which completes the estimate for $J_2$. 

\medskip
\subsubsection*{The estimate for $J_{3}$:} 
If $\Q$ is  such that $(k,j)\in \mathscr{H}$, then $\Q$ does not
contain the origin.  Since it is a dyadic cube, we have that
 $\dist(\Q,0) \geq \ell(\Q)$. Therefore,  there exists a constant
 $R>1$ depending only on $n$  such that 
\begin{equation}\label{H_cube_estimate}
\sup_{x\in \Q} |x| \leq R \inf_{x\in \Q} |x|. 
\end{equation}

\begin{remark} \label{remark:grids}
The estimate $\dist(\Q,0) \geq \ell(\Q)$ holds because we are working
with the grid $\D_0$.  For an arbitrary grid $\D_t$,  since the origin
will be contained in one of the cubes $P_i$,  we will have that for
some $c>0$, $\dist(\Q,0) \geq c\ell(\Q)$, and so
\eqref{H_cube_estimate} will hold with a possibly larger constant $R$.
\end{remark}

\medskip

By the continuity of $p(\cdot)$, 
there exists $x_+$ in the closure of $\Q$ such
that $p_{+}(\Q)=p(x_{+})$.  Hence, since $\pp \in LH$, for all $x\in
\Q$, by~\eqref{H_cube_estimate},
\begin{multline} \label{eqn:p-J3-est}
0 \leq p_{+}(\Q)-p(x)
\leq  |p(x_{+})-p(x)|+ |p(x)-p_{\infty}| \\
\leq \frac{C_{\infty}}{\log(e+|x_{+}|)}
+\frac{C_{\infty}}{\log(e+|x|)}
 \lesssim \frac{1}{\log(e+|x|)}.
\end{multline}
In the same way, for $l=1,2$ we have that $p_l(\cdot)$ satisfies
\begin{equation} \label{eqn:p2-J3-est}
|(p_{l})_{-}(\Q)-p_{l}(x)|
\lesssim \frac{1}{\log(e+|x|)}.
\end{equation}

\smallskip

To estimate $J_{3}$ we need to divide $\mathscr{H}$ into two subsets
depending on the size of the cubes~$\Q$ with respect to $\sigma_2$:
\[ 
  \mathscr{H}_{1}=\lbrace (k,j) \in \mathscr{H}:
  \sigma_2(\Q)\leq 1 \rbrace, \quad
\mathscr{H}_{2}=\lbrace (k,j) \in \mathscr{H}: \sigma_2(\Q)>1
\rbrace.
\]
We first estimate the sum over $\mathscr{H}_{1}$.
By~\eqref{eqn:p-J3-est} and Lemmas~\ref{lemma:p-infty-px}
and~\ref{lemma:infty-bound}, 
\begin{align*}
 &\sum_{(k,j)\in \mathscr{H}_{1}} \int_{E_{j}^{k}}
\prod_{l=1,4}\langle h_{l}\sigma_{\rho(l)} \rangle_{\Q}^{p(x)}
   u(x)\,dx \\
 & \qquad \lesssim \sum_{(k,j)\in \mathscr{H}_{1}} \int_{E_{j}^{k}}
\prod_{l=1,4}\langle h_{l}\sigma_{\rho(l)}
   \rangle_{\Q}^{p_{+}(\Q)} u(x)\,dx
+\sum_{(k,j)\in \mathscr{H}_{1}}\int_{E_{j}^{k}}\frac{u(x)}{(e+|x|)^{tnp_{-}}}\,dx\\
   & \qquad  \leq \sum_{(k,j)\in \mathscr{H}_{1}}
     \int_{E_{j}^{k}}\prod_{l=1,4}\langle h_{l}\sigma_{\rho(l)}
  \rangle_{\Q}^{p_{+}(\Q)} u(x)\,dx + 1.
\intertext{
By Lemma~\ref{lemma:diening}, \eqref{Bound_function_geq1}, and
     \eqref{EQ-I1estimte}, and since $h_1\geq 1$, $h_4\leq 1$ and $\sigma_2(\Q)\leq 1$,}
 &\qquad = \sum_{(k,j)\in \mathscr{H}_{1}} \int_{E_{j}^{k}}
\bigg(\int_{\Q}{h_{1}\sigma_{1}\,dy}\bigg)^{p_{+}(\Q)}
\bigg(\frac{1}{\sigma_{2}(\Q)}\int_{\Q}{h_{4}\sigma_{2}\,dy}\bigg)^{p_{+}(\Q)}\\
& \qquad \qquad \qquad \qquad \times
  |\Q|^{-2p_{+}(\Q)}\sigma_{2}(\Q)^{p_{+}(\Q)}u(x)\,dx +1; \\
& \qquad \lesssim \sum_{(k,j)\in \mathscr{H}_{1}} 
\langle h_{1} \rangle_{\sigma_{1},\Q}^{q(\Q)}
\langle h_{4}\rangle_{\sigma_{2},\Q}^{q(\Q)}
\int_{E_{j}^{k}}|\Q|^{-2p(x)}
\sigma_{1}(\Q)^{q(\Q)}\sigma_{2}(\Q)^{q(\Q)} u(x)\,dx + 1\\
& \qquad \lesssim \sum_{(k,j)\in \mathscr{H}_{1}}
\prod_{l=1,4}\langle h_{l}\rangle_{\sigma_{\rho(l),\Q}}^{q(\Q)}
\sigma_{1}(\Q)^{\frac{q(\Q)}{(p_{1})_{-}(\Q)}}\sigma_{2}(\Q)^{\frac{q(\Q)}{(p_{2})_{-}(\Q)}}
  + 1\\
&\qquad \leq  \sum_{(k,j)\in \mathscr{H}_{1}}
\langle   h_{1}^{\frac{p_{1}(\cdot)}{(p_{1})_{-}(\Q)}} \rangle_{\sigma_{1},\Q}^{q(\Q)}
\sigma_{1}(\Q)^{\frac{q(\Q)}{(p_{1})_{-}(\Q)}}
\langle h_{4}\rangle_{\sigma_{2},\Q}^{q(\Q)}
\sigma_{2}(\Q)^{\frac{q(\Q)}{(p_{2})_{-}(\Q)}} + 1;\\
\intertext{by H\"older's inequality and Young's inequality,}
&\qquad \leq \sum_{(k,j)\in \mathscr{H}_{1}}
\langle h_{1}^{\frac{p_{1}(\cdot)}{(p_{1})_{-}}}
  \rangle_{\sigma_{1},\Q}^{(p_{1})_{-}\frac{q(\Q)}{(p_{1})_{-}(\Q)}}
\sigma_{1}(\Q)^{\frac{q(\Q)}{(p_{1})_{-}(\Q)}}
\langle h_{4}\rangle_{\sigma_{2},\Q}^{q(\Q)}
\sigma_{2}(\Q)^{\frac{q(\Q)}{(p_{2})_{-}(\Q)}} +1\\
& \qquad \lesssim \sum_{(k,j)\in \mathscr{H}_{1}}
\langle
  h_{1}^{\frac{p_{1}(\cdot)}{(p_{1})_{-}}}\rangle_{\sigma_{1},\Q}^{(p_{1})_{-}}
\sigma_{1}(\Q)
+ \sum_{(k,j)\in \mathscr{H}_{1}}
\langle h_{4}\rangle_{\sigma_{2},\Q}^{(p_{2})_{-}(\Q)}
\sigma_{2}(\Q) + 1\\
& \qquad = K_{1}+K_{2} +1.
\end{align*}

The proof that $K_1$ is bounded is exactly the same  as the
final estimate for $I_1$, beginning at \eqref{eqn:final-I1-est}.
Therefore, to complete the estimate for the sum over
$\mathscr{H}_{1}$, we need to bound $K_2$.  
By Lemmas~\ref{lemma:p-infty-px} and ~\ref{lemma:infty-bound} (applied
twice) and by Lemma~\ref{lemma:wtd-max-bound},
\begin{align*}
K_{2}
& \lesssim  \sum_{(k,j)\in \mathscr{H}_{1}} \int_{E_{j}^{k}}
\langle
  h_{4}\rangle_{\sigma_{2},\Q}^{(p_{2})_{-}(\Q)}\sigma_{2}(x)\,dx \\
&\lesssim \sum_{(k,j)\in \mathscr{H}_{1}}
 \int_{E_{j}^{k}}\langle
  h_{4}\rangle_{\sigma_{2},\Q}^{(p_{2})_{\infty}}
\sigma_{2}(x)\,dx
+ \sum_{(k,j)\in \mathscr{H}_{1}}
 \int_{E_{j}^{k}}\frac{\sigma_{2}(x)}{(e+|x|)^{nt(p_{2})_{-}}}\,dx \\
&\leq \int_{\subRn}
  {M}_{\sigma_{2}}^{d}h_{4}(x)^{(p_{2})_{\infty}}\sigma_{2}(x)\,dx
 +
\int_{\subRn}\frac{\sigma_{2}(x)}{(e+|x|)^{nt(p_{2})_{-}}}\,dx\\
&\lesssim \int_{\subRn}h_{4}(x)^{(p_{2})_{\infty}}\sigma_{2}(x)\,dx + 1\\
& \lesssim \int_{\subRn}h_{4}(x)^{p_{2}(x)}\sigma_{2}(x)\,dx
+\int_{\subRn}\frac{\sigma_{2}(x)}{(e+|x|)^{nt(p_{2})_{-}}}\,dx + 1\\
& \lesssim 1.
\end{align*}

\medskip

To estimate the sum over $\mathscr{H}_{2}$, first note that by
Lemma~\ref{lemma:gen-holder} we have that
\begin{multline}\label{H2-first-estimate}
\int_{\Q}{h_{1}\sigma_{1}\,dy}
\lesssim \|h_{1}\|_{L_{\sigma_{1}}^{\pap}} 
\|\chi_{\Q}\|_{L_{\sigma_{1}}^{\cpap}}\\
\lesssim \|f_{1}\|_{L_{\sigma_{1}}^{\pap}}
\|w_{1}^{-1}\chi_{\Q}\|_{\cpap}
\leq c_0\|w_{1}^{-1}\chi_{\Q}\|_{\cpap};
\end{multline}
similarly, we have that
\begin{equation}\label{H2-second-estimate}
\int_{\Q}{h_{4}\sigma_{2}\,dy}
\leq c_0\|w_{2}^{-1}\chi_{\Q}\|_{\cpbp}.
\end{equation}

We now divide the cubes in $\mathscr{H}_2$ into two
subsets depending on the size of $\sigma_1(\Q)$:
\[ \mathscr{H}_{2a} = \{ (k,j)\in \mathscr{H}_2 : \sigma_1(\Q) \geq 1 \},
  \quad
\mathscr{H}_{2b} = \{ (k,j)\in \mathscr{H}_2 : \sigma_1(\Q)< 1 \}.  \]
We first estimate the sum over $\mathscr{H}_{2a}$. 
Given  \eqref{H2-first-estimate} and  \eqref{H2-second-estimate}, by
Lemma~\ref{lemma:p-infty-px},
\begin{align*}
 &\sum_{(k,j)\in \mathscr{H}_{2a}} \int_{E_{j}^{k}}\prod_{l=1,4}
\langle h_{l}\sigma_{\rho(l)}\rangle_{\Q}^{p(x)}u(x)\,dx \\
 & \quad \quad  \leq c_0^{2p_+} \sum_{(k,j)\in \mathscr{H}_{2a}}
   \int_{E_{j}^{k}} \prod_{l=1,4}
\bigg(c_0^{-1}\|w_{\rho(l)}^{-1}\chi_{\Q}\|_{p_{\rho(l)}^{\prime}(\cdot)}^{-1}
\int_{\Q}h_{l}\sigma_{\rho(l)}\,dy \bigg)^{p(x)} \\
& \qquad \qquad \qquad \qquad \qquad 
\times \prod_{l=1}^{2}
\bigg(\frac{\|w_{l}^{-1}\chi_{\Q}\|_{p_{l}^{\prime}(\cdot)}}{|\Q|}\bigg)^{p(x)}u(x)\,dx\\
 &  \quad \quad \lesssim \sum_{(k,j)\in
   \mathscr{H}_{2a}}\int_{E_{j}^{k}}
\prod_{l=1,4}\bigg(c_0^{-1}\|w_{\rho(l)}^{-1}\chi_{\Q}\|_{p_{\rho(l)}^{\prime}(\cdot)}^{-1}
\int_{\Q}h_{l}\sigma_{\rho(l)}\,dy\bigg)^{p_\infty} \\
& \qquad \qquad \qquad \qquad \qquad 
\times 
\prod_{l=1}^{2}\bigg(\frac{\|w_{l}^{-1}\chi_{\Q}\|_{p_{l}^{\prime}(\cdot)}}
{|\Q|}\bigg)^{p(x)}u(x)\,dx \\
 & \qquad \qquad \qquad  + 
\sum_{(k,j)\in \mathscr{H}_{2a}} \int_{E_{j}^{k}} \prod_{l=1}^{2}
\bigg(\frac{\|w_{l}^{-1}\chi_{\Q}\|_{p_{l}^{\prime}(\cdot)}}{|\Q|}\bigg)^{p(x)}
\frac{u(x)}{(e+|x|)^{tnp_{-}}}\,dx \\
 & \quad \quad =L_{1}+ L_{2}.
\end{align*}

We first estimate $L_2$. Since
$\sigma_2(E_j^k)\gtrsim\sigma_{2}(\Q)\geqslant 1$, by
\eqref{H_cube_estimate}, \eqref{eqn:modular-Ap} and
Lemma~\ref{lemma:infty-bound},
 \begin{align*}
         L_{2}
& \leq \sum_{(k,j)\in \mathscr{H}_{2a}}\sup_{x\in
  \Q}(e+|x|)^{-ntp_{-}}
\int_{\Q}{\prod_{l=1}^{2}\|w_{l}^{-1}\chi_{\Q}\|_{p_{l}^{\prime}(\cdot)}^{p(x)}
|\Q|^{-2p(x)}u(x)\,dx}\\
  & \lesssim \sum_{(k,j)\in \mathscr{H}_{2a}}\inf_{x\in \Q}(e+|x|)^{-ntp_{-}}\sigma_{2}(E_j^k)\\
 & \lesssim \int_{\subRn}\frac{\sigma_{2}(x)}{(e+|x|)^{ntp_{-}}}\,dx \\
 & \lesssim 1.
 \end{align*}
              
\smallskip
    
In order to estimate $L_{1}$ we first note that for $l=1,2$, since
$\sigma_l(\Q)\geq 1$, by
Lemma~\ref{lemma:p-infty-cond}, 
\begin{equation} \label{eqn:p-infty-Hest}
      \bigg(\frac{\sigma_{l}(\Q)}
{\|w_{l}^{-1}\chi_{\Q}\|_{p_l'(\cdot)}}\bigg)^{p_{\infty}}
 \lesssim
 \bigg(\frac{\sigma_{l}(\Q)}
{\sigma_{l}(\Q)^{\frac{1}{(p_{l}^{\prime})_{\infty}}}}\bigg)^{p_{\infty}}
=  \sigma_{l}(\Q)^{\frac{p_{\infty}}{(p_{l})_{\infty}}}. 
\end{equation}
Given this estimate, by~\eqref{eqn:modular-Ap} and Young's inequality
we have that
 \begin{align*}
    L_{1}&\lesssim \sum_{(k,j)\in
           \mathscr{H}_{2a}}\int_{E_{j}^{k}}{\prod_{l=1,4}
\langle h_{l} \rangle_{\sigma_{\rho(l)},\Q}^{p_{\infty}}
\sigma_{1}(\Q)^{\frac{p_{\infty}}{(p_{1})_{\infty}}}
\sigma_{2}(\Q)^{\frac{p_{\infty}}{(p_{2})_{\infty}}}}\\
    & \qquad \qquad \qquad \qquad  \times
      \prod_{l=1}^{2}\bigg(\frac{\|w_{l}^{-1}\chi_{\Q}
      \|_{p_{l}^{\prime}(\cdot)}}{|\Q|}\bigg)^{p(x)}u(x)\,dx \\
    & \leq \sum_{(k,j)\in \mathscr{H}_{2a}}
\prod_{l=1,4} \langle h_{l} \rangle_{\sigma_{\rho(l)},\Q}^{p_{\infty}}
\sigma_{1}(\Q)^{\frac{p_{\infty}}{(p_{1})_{\infty}}}
\sigma_{2}(\Q)^{\frac{p_{\infty}}{(p_{2})_{\infty}}}\\
    &  \qquad \qquad \qquad \qquad
      \times\int_{\Q}{\prod_{l=1}^{2}\|w_{l}^{-1}\chi_{\Q} 
\|_{p_{l}^{\prime}(\cdot)}^{p(x)}|\Q|^{-2p(x)}u(x)\,dx}\\
    &\lesssim \sum_{(k,j)\in \mathscr{H}_{2a}}
\langle h_{1}\rangle_{\sigma_{1},\Q}^{(p_{1})_{\infty}}\sigma_{1}(\Q)
+\sum_{(k,j)\in \mathscr{H}_{2a}}\langle
      h_{4}\rangle_{\sigma_{2},\Q}^{(p_{2})_{\infty}}
\sigma_{2}(\Q).\\
\end{align*}     
The estimate of the last term is identical to the estimate for $J_2$
above, beginning at inequality~\eqref{eqn:J2-final-est}; here we
use the fact that $\sigma_1(\Q)\geq 1$ to
get~\eqref{sigma1_Averange_Bound}.    

\smallskip

The estimate over $\mathscr{H}_{2b}$ is similar, but we must replace the
exponent $p_\infty$ with $r(\Q)$, which is defined by
\[ \frac{1}{r(\Q)} = \frac{1}{(p_1)_-(\Q)} + \frac{1}{(p_2)_\infty}.  \]
Then by~\eqref{eqn:p2-J3-est}, for $x\in \Q$,
\[ \bigg|\frac{1}{p(x)}-\frac{1}{r(\Q)}\bigg|
\leq \bigg|\frac{1}{p_1(x)}-\frac{1}{(p_1)_-(\Q)}\bigg|
+ \bigg|\frac{1}{p_2(x)}-\frac{1}{(p_2)_\infty}\bigg| 
\lesssim \frac{1}{\log(e+|x|)}.  \]
We can then argue as we did for the sum over $\mathscr{H}_{2a}$ above
to get
\begin{align*}
 &\sum_{(k,j)\in \mathscr{H}_{2b}} \int_{E_{j}^{k}}\prod_{l=1,4}
\langle h_{l}\sigma_{\rho(l)}\rangle_{\Q}^{p(x)}u(x)\,dx \\
 &  \quad \quad \lesssim \sum_{(k,j)\in
   \mathscr{H}_{2b}}\int_{E_{j}^{k}}
\prod_{l=1,4}\bigg(c_0^{-1}\|w_{\rho(l)}^{-1}\chi_{\Q}\|_{p_{\rho(l)}^{\prime}(\cdot)}^{-1}
\int_{\Q}h_{l}\sigma_{\rho(l)}\,dy\bigg)^{r(\Q)} \\
& \qquad \qquad \qquad \qquad \qquad 
\prod_{l=1}^{2}\bigg(\frac{\|w_{l}^{-1}\chi_{\Q}\|_{p_{l}^{\prime}(\cdot)}}
{|\Q|}\bigg)^{p(x)}u(x)\,dx \\
 & \qquad \qquad \qquad  + 
\sum_{(k,j)\in \mathscr{H}_{2b}} \int_{E_{j}^{k}} \prod_{l=1}^{2}
\bigg(\frac{\|w_{l}^{-1}\chi_{\Q}\|_{p_{l}^{\prime}(\cdot)}}{|\Q|}\bigg)^{p(x)}
\frac{u(x)}{(e+|x|)^{tnp_{-}}}\,dx \\
 & \quad \quad =M_{1}+ M_{2}.
\end{align*}

The estimate for $M_2$ is identical to the estimate for $L_2$.  To
estimate $M_1$, we again use \eqref{eqn:p-infty-Hest} for $\sigma_2$,
replacing $p_\infty$ with $r(\Q)$.  Because $\sigma_1(\Q)<1$ we need
to replace~\eqref{eqn:p-infty-Hest} with a different estimate. Since
$(p_1^{'})_{\pm}(\Q)=(p_1)_{\mp}(\Q)^{'}$, by the estimate
\eqref{eq:sigma-w-estimate}, replacing $q(\Q)$ with $r(\Q)$, we get
\[ \bigg(\frac{\sigma_1(\Q)}{\|w_1^{-1}\chi_\Q\|_{p_1'(\cdot)}}\bigg)^{r(\Q)}
\lesssim \sigma_1(\Q)^{\frac{r(\Q)}{(p_1)_-(\Q)}}. \]

We can now modify the estimate for $L_1$ to estimate $M_1$:
\begin{align*}
    M_{1}&\lesssim \sum_{(k,j)\in
           \mathscr{H}_{2b}}\int_{E_{j}^{k}}{\prod_{l=1,4}
\langle h_{l} \rangle_{\sigma_{\rho(l)},\Q}^{r(\Q)}
\sigma_{1}(\Q)^{\frac{r(\Q)}{(p_1)_-(\Q)}}
\sigma_{2}(\Q)^{\frac{r(\Q)}{(p_{2})_{\infty}}}}\\
    & \qquad \qquad \qquad \qquad  \times
      \prod_{l=1}^{2}\bigg(\frac{\|w_{l}^{-1}\chi_{\Q}
      \|_{p_{l}^{\prime}(\cdot)}}{|\Q|}\bigg)^{p(x)}u(x)\,dx \\
    & \leq \sum_{(k,j)\in \mathscr{H}_{2b}}
\prod_{l=1,4}\langle h_{l} \rangle_{\sigma_{\rho(l)},\Q}^{r(\Q)}
\sigma_{1}(\Q)^{\frac{r(\Q)}{(p_1)_-(\Q)}}
\sigma_{2}(\Q)^{\frac{r(\Q)}{(p_{2})_{\infty}}}\\
    &  \qquad \qquad \qquad \qquad
      \times\int_{\Q}{\prod_{l=1}^{2}\|w_{l}^{-1}\chi_{\Q} 
\|_{p_{l}^{\prime}(\cdot)}^{p(x)}|\Q|^{-2p(x)}u(x)\,dx}\\
    &\lesssim \sum_{(k,j)\in \mathscr{H}_{2b}}
\langle h_{1}\rangle_{\sigma_{1},\Q}^{(p_{1})_-(\Q)}\sigma_{1}(\Q)
+\sum_{(k,j)\in \mathscr{H}_{2b}}\langle
      h_{4}\rangle_{\sigma_{2},\Q}^{(p_{2})_{\infty}}
\sigma_{2}(\Q).\\
\end{align*}     
The estimate for the second term in the last line is the same as the
final estimate for $J_2$; we use the same argument above to estimate
$L_1$.  The estimate for the first term is the same as the estimate
for $K_1$ above, noting that since $h_1\geq 1$ and by H\"older's
inequality,
\[ \langle h_{1}\rangle_{\sigma_{1},\Q}^{(p_{1})_-(\Q)}
\leq \langle
h_{1}^{\frac{\pap}{(p_{1})_-(\Q)}}\rangle_{\sigma_{1},\Q}^{(p_{1})_-(\Q)}
\leq \langle
h_{1}^{\frac{\pap}{(p_{1})_-}}\rangle_{\sigma_{1},\Q}^{(p_{1})_-}. \]
This
completes the estimate of $M_1$ and so of $I_2$.

\begin{remark}
As noted above, the argument for $I_3$ is the same as that for $I_2$, 
replacing $h_1\sigma_1$ with $h_3\sigma_2$ and $h_4\sigma_2$ with
$h_2\sigma_1$. 
\end{remark}

\medskip    
     
 \subsection*{The estimate for $I_4$:}
The estimate for $I_4$ parallels that for $I_2$; in particular we will
decompose $I_4$ into essentially the same parts as we did above.
For some parts the estimate is very similar to the
 corresponding part $I_2$, and so we give the key inequalities but
 will omit some of the details.    For other parts we will need to
 modify the argument and we will present these in more detail. 

 Begin by forming the bilinear Calder\'on-Zygmund cubes associated
 with $\M^d(h_2\sigma_1,h_4\sigma_2)$.  We then decompose the
 collection of these cubes into the sets $\mathscr{F}$, $\mathscr{G}$
 and $\mathscr{H}$, defined as above.  Denote the sums over these sets
 by $N_1$, $N_2$ and $N_3$.

\subsubsection*{The estimate for $N_{1}$:} 
The estimate for $N_1$  is very similar to that for $J_1$ above.  We
replace the arguments used for the $h_1$ term and estimate the $h_2$
term and the $h_4$ term in the same way, using the fact that
$h_2,\,h_4 \leq 1$:
\begin{align*}
N_{1}
&= \sum_{(k,j)\in \mathscr{F}}\int_{E_{j}^{k}}{\prod_{l=2,4}\langle h_{l}\sigma_{\rho(l)}\rangle_{\Q}^{p(x)}u(x)\,dx}\\
&\leq \sum_{(k,j)\in \mathscr{F}}\int_{E_{j}^{k}}{\prod_{l=1}^{2}\langle \sigma_{l}\rangle_{\Q}^{p(x)}u(x)\,dx}\\
&= \sum_{(k,j)\in \mathscr{F}}\int_{E_{j}^{k}}
\prod_{l=1}^{2}\sigma_{l}(\Q)^{p(x)-q(\Q)}
\sigma_{1}(\Q)^{q(\Q)}\sigma_{2}(\Q)^{q(\Q)}|\Q|^{-2p(x)}u(x)\,dx \\
 &\leq \sum_{(k,j)\in \mathscr{F}}
\prod_{l=1}^{2}\bigg(1+\sigma_{l}(\Q) \bigg)^{p_{+}(\Q)-q(\Q)}
\int_{E_{j}^{k}}\sigma_{1}(\Q)^{q(\Q)}\sigma_{2}(\Q)^{q(\Q)}|\Q|^{-2p(x)}u(x)\,dx\\
 &\lesssim \prod_{l=1}^{2}\bigg(1+\sigma_{l}(P)\bigg)^{p_{+}-q_{-}}
\sum_{(k,j)\in \mathscr{F}}
\sigma_{1}(\Q)^{\frac{q(\Q)}{(p_{1})_{-}(\Q)}}\sigma_{2}(\Q)^{\frac{q(\Q)}{(p_{2})_{-}(\Q)}}\\
 &\lesssim\sum_{(k,j)\in \mathscr{F}}\sigma_{1}(\Q)
+ \sum_{(k,j)\in \mathscr{F}}\sigma_{2}(\Q)\\
 &\lesssim \sum_{(k,j)\in \mathscr{F}}\sigma_{1}(E_{j}^{k})
+\sum_{(k,j)\in \mathscr{F}}\sigma_{2}(E_{j}^{k})\\
 &\leq \sigma_{1}(P)+\sigma_{2}(P) \\
& \lesssim 1.
\end{align*}

\subsubsection*{The estimate for $N_{2}$:} 
To estimate $N_2$ we modify the argument for $J_2$.  By  the
definition of $\A_{\vec{p}(\cdot)}$ and by
Lemma~\ref{lemma:gen-holder} 
we have  that
   \begin{align*}
    &\frac{1}{|\Q|^{2}}\int_{\Q}{h_{2}\sigma_{1}\,dy}\int_{\Q}{h_{4}\sigma_{2}\,dy} \\
    & \quad \lesssim \|w\chi_{\Q}\|_{\pp}^{-1}
\prod_{l=1}^{2}
\|w_{l}^{-1}\chi_{\Q}\|_{p_{l}^{\prime}(\cdot)}^{-1}
\|h_{2}\|_{L_{\sigma_{1}}^{\pap}}\|h_{4}\|_{L_{\sigma_{2}}^{\pbp}}
\|\chi_{\Q}\|_{L_{\sigma_{1}}^{\cpap}}\|\chi_{\Q}\|_{L_{\sigma_{2}}^{\cpbp}}\\
    & \quad  =
      \|w\chi_{\Q}\|_{\pp}^{-1}\prod_{l=1}^{2}\|w_{l}^{-1}\chi_{\Q}\|_{p_{l}^{\prime}(\cdot)}^{-1}\|h_{2}\|_{L_{\sigma_{1}}^{\pap}}\|h_{4}\|_{L_{\sigma_{2}}^{\pbp}}\|w_{1}^{-1}\chi_{\Q}\|_{\cpap}
      \|w_{2}^{-1}\chi_{\Q}\|_{\cpbp};\\
\intertext{since
     $u(Q)\geq u(P_i)\geq 1$, by Lemma~\ref{lemma:mod-norm},}
    & \quad \lesssim\|w\chi_{\Q}\|_{\pp}^{-1} \\
   & \quad \leq c_0.
       \end{align*}
Therefore, by Lemma~\ref{lemma:p-infty-px},
   \begin{align*}
 N_{2}
&= \sum_{(k,j)\in \mathscr{G}}\int_{E_{j}^{k}}
\prod_{l=2,4}\langle h_{l}\sigma_{\rho(l)}\rangle_{\Q}^{p(x)}u(x)\,dx \\
  & \lesssim \sum_{(k,j)\in \mathscr{G}}
\int_{E_{j}^{k}}\bigg(\frac{c_0^{-1}}{|\Q|^{2}}
\int_{\Q}h_{2}\sigma_{1}\,dy \int_{\Q}h_{4}\sigma_{2}\,dy\bigg)^{p(x)}u(x)\,dx\\
  &\lesssim \sum_{(k,j)\in \mathscr{G}}
 \int_{E_{j}^{k}} \prod_{l=2,4}\langle
    h_{l}\sigma_{\rho(l)}\rangle_{\Q}^{p_{\infty}}u(x)\,dx
+ \sum_{(k,j)\in \mathscr{G}}\int_{E_{j}^{k}}\frac{u(x)}{(e+|x|)^{tnp_{-}}}\,dx.
   \end{align*}
By Lemma~\ref{lemma:infty-bound}, the second term on the last line is
bounded by a constant $1$.  We
estimate the first term using~\eqref{P_infty_Bounded}:
       \begin{align*}
       &\sum_{(k,j)\in \mathscr{G}}\int_{E_{j}^{k}}\prod_{l=2,4}
\langle h_{l}\sigma_{\rho(l)}\rangle_{\Q}^{p_{\infty}}u(x)\,dx \\
      & \quad  = \sum_{(k,j)\in
        \mathscr{G}}\int_{E_{j}^{k}}\prod_{l=2,4}
\langle h_{l}\rangle_{\sigma_{\rho(l)},\Q}^{p_{\infty}}
\sigma_{1}(\Q)^{p_{\infty}}\sigma_{2}(\Q)^{p_{\infty}}|\Q|^{-2p_{\infty}}u(x)\,dx
         \\
      & \quad \lesssim \sum_{(k,j)\in \mathscr{G}}
\langle h_{2}\rangle_{\sigma_{1},\Q}^{p_{\infty}}
\sigma_{1}(\Q)^{\frac{p_{\infty}}{(p_{1})_{\infty}}}
\langle h_{4}\rangle_{\sigma_{2},\Q}^{p_{\infty}}
\sigma_{2}(\Q)^{\frac{p_{\infty}}{(p_{2})_{\infty}}};\\
\intertext{By Young's inequality and
         Lemmas~\ref{lemma:wtd-max-bound},~\ref{lemma:p-infty-px},
         and~\ref{lemma:infty-bound},}
 & \quad \lesssim \sum_{(k,j)\in \mathscr{G}}
\langle h_{2}\rangle_{\sigma_{1},\Q}^{(p_{1})_{\infty}}\sigma_{1}(\Q)
+  \sum_{(k,j)\in \mathscr{G}}
\langle h_{4}\rangle_{\sigma_{2},\Q}^{(p_{2})_{\infty}}\sigma_{2}(\Q)\\
     & \quad \lesssim  \sum_{(k,j)\in \mathscr{G}}
\langle
       h_{2}\rangle_{\sigma_{1},\Q}^{(p_{1})_{\infty}}\sigma_{1}(E_{j}^{k})
+  \sum_{(k,j)\in \mathscr{G}}
\langle
       h_{4}\rangle_{\sigma_{2},\Q}^{(p_{2})_{\infty}}\sigma_{2}(E_{j}^{k})\\ 
   & \quad \leq
     \int_{\subRn}{M}_{\sigma_{1}}^{d}h_{2}(x)^{(p_{1})_{\infty}}\sigma_{1}(x)\,dx
+\int_{\subRn}{M}_{\sigma_{2}}^{d}h_{4}(x)^{(p_{2})_{\infty}}\sigma_{2}(x)\,dx\\
    & \quad \lesssim
      \int_{\subRn}h_{2}(x)^{(p_{1})_{\infty}}\sigma_{1}(x)\,dx
+ \int_{\subRn}h_{4}(x)^{(p_{2})_{\infty}}\sigma_{2}(x)\,dx\\
    & \quad \lesssim \int_{\subRn}h_{2}(x)^{p_{1}(x)}\sigma_{1}(x)\,dx
+\int_{\subRn}h_{4}(x)^{p_{2}(x)}\sigma_{2}(x)\,dx\\
      &\quad \qquad \quad +
        \int_{\subRn}\frac{\sigma_{1}(x)}{(e+|x|)^{tn(p_{1})_{-}}}\,dx
+\int_{\subRn}\frac{\sigma_{2}(x)}{(e+|x|)^{tn(p_{2})_{-}}}\,dx \\
& \quad \lesssim 1.
       \end{align*}

 \subsubsection*{The estimate for $N_{3}$:} 
The estimate for $N_3$ is broadly similar to the estimate for $J_3$
above, but it differs considerably in the details.   We first begin by
dividing the cubes in $\mathscr{H}$ into the sets  $\mathscr{H}_1$
and  $\mathscr{H}_2$ as before.  However, we now have to subdivide both of these
sets and not just $\mathscr{H}_2$.  Define
\[ \mathscr{H}_{1a} = \{ (k,j) \in \mathscr{H}_{1} : \sigma_1(\Q)\leq
  1, \sigma_2(\Q) \leq 1 \} \]
and
\[ \mathscr{H}_{1b} = \{ (k,j) \in \mathscr{H}_{1} : \sigma_1(\Q)>
  1, \sigma_2(\Q) \leq 1\}. \]

\medskip

The estimate for the sum over $\mathscr{H}_{1a}$ is similar to the
estimate over $\mathscr{H}_{1}$ above for $J_3$, but we use the fact
that both $h_2,\,h_4\leq 1$.  By Lemmas~\ref{lemma:p-infty-px}
and~\ref{lemma:infty-bound}, 
\begin{align*}
 &\sum_{(k,j)\in \mathscr{H}_{1a}} \int_{E_{j}^{k}}
\prod_{l=2,4}\langle h_{l}\sigma_{\rho(l)} \rangle_{\Q}^{p(x)}
   u(x)\,dx \\
 & \qquad \lesssim \sum_{(k,j)\in \mathscr{H}_{1a}} \int_{E_{j}^{k}}
\prod_{l=2,4}\langle h_{l}\sigma_{\rho(l)}
   \rangle_{\Q}^{p_{+}(\Q)} u(x)\,dx
+\sum_{(k,j)\in \mathscr{H}_{1a}}\int_{E_{j}^{k}}\frac{u(x)}{(e+|x|)^{tnp_{-}}}\,dx\\
   & \qquad  \leq \sum_{(k,j)\in \mathscr{H}_{1a}}
    \prod_{l=2,4}\langle h_{l}
     \rangle_{\sigma_{\rho(l)},\Q}^{p_{+}(\Q)}  
\int_{E_{j}^{k}}|\Q|^{-2 p_{+}(\Q)} \prod_{l=2,4}\sigma_{\rho(l)}(\Q)^{p_{+}(\Q)} u(x)\,dx + 1.
\intertext{Since $h_2,\, h_4\leq 1$ and $\sigma_l(\Q)\leq 1$,
     $l=1,\,2$, by \eqref{EQ-I1estimte}, replacing
     $|Q_j^k|^{-2p(x)}$ with $|Q_j^k|^{-2p_+(Q_j^k)}$ (which we can do
     by Lemma~\ref{lemma:diening}),}
  & \qquad  \leq \sum_{(k,j)\in \mathscr{H}_{1a}}
    \prod_{l=2,4}\langle h_{l}
     \rangle_{\sigma_{\rho(l)},\Q}^{q(\Q)}  
\int_{E_{j}^{k}}|\Q|^{-2 p_{+}(\Q)}
    \prod_{l=2,4}\sigma_{\rho(l)}(\Q)^{q(\Q)} u(x)\,dx + 1 \\
& \qquad \lesssim \sum_{(k,j)\in \mathscr{H}_{1a}}
\prod_{l=2,4}\langle h_{l}\rangle_{\sigma_{\rho(l),\Q}}^{q(\Q)}
\sigma_{1}(\Q)^{\frac{q(\Q)}{(p_{1})_{-}(\Q)}}\sigma_{2}(\Q)^{\frac{q(\Q)}{(p_{2})_{-}(\Q)}}
  + 1;\\
\intertext{by Young's inequality,}
& \qquad \lesssim \sum_{(k,j)\in \mathscr{H}_{1a}}
\langle
  h_{2}\rangle_{\sigma_{1},\Q}^{(p_{1})_{-}}
\sigma_{1}(\Q)
+ \sum_{(k,j)\in \mathscr{H}_{1a}}
\langle h_{4}\rangle_{\sigma_{2},\Q}^{(p_{2})_{-}(\Q)}
\sigma_{2}(\Q) + 1.
\end{align*}
Both of the final terms are estimated as $K_2$ above.

\medskip

To estimate the sum over $\mathscr{H}_{1b}$, we first define the
exponent $s(\Q)$ by
  \begin{equation*}
  \frac{1}{s(\Q)}=\frac{1}{(p_1)_{\infty}}+\frac{1}{(p_2)_{+}(\Q)}.
  \end{equation*}
Then, arguing as we did for~\eqref{eqn:p-J3-est}, we get that for $x\in \Q$, 
\[ \bigg|\frac{1}{p(x)}-\frac{1}{s(\Q)}\bigg|
\leq \bigg|\frac{1}{p_1(x)}-\frac{1}{(p_1)_\infty}\bigg|
+ \bigg|\frac{1}{p_2(x)}-\frac{1}{(p_2)_+(\Q)}\bigg| 
\lesssim \frac{1}{\log(e+|x|)}.  \]
Given this, by~\eqref{H2-first-estimate} (for $h_2$ instead of $h_1$),~\eqref{H2-second-estimate}
and Lemma~\ref{lemma:p-infty-px},
\begin{align*}
 &\sum_{(k,j)\in \mathscr{H}_{1b}} \int_{E_{j}^{k}}\prod_{l=2,4}
\langle h_{l}\sigma_{\rho(l)}\rangle_{\Q}^{p(x)}u(x)\,dx \\
 &  \quad \quad \lesssim \sum_{(k,j)\in
   \mathscr{H}_{1b}}\int_{E_{j}^{k}}
\prod_{l=2,4}\bigg(c_0^{-1}\|w_{\rho(l)}^{-1}\chi_{\Q}\|_{p_{\rho(l)}^{\prime}(\cdot)}^{-1}
\int_{\Q}h_{l}\sigma_{\rho(l)}\,dy\bigg)^{s(\Q)} \\
& \qquad \qquad \qquad \qquad \qquad 
\times 
\prod_{l=1}^{2}\bigg(\frac{\|w_{l}^{-1}\chi_{\Q}\|_{p_{l}^{\prime}(\cdot)}}
{|\Q|}\bigg)^{p(x)}u(x)\,dx \\
 & \qquad \qquad \qquad  + 
\sum_{(k,j)\in \mathscr{H}_{1b}} \int_{E_{j}^{k}} \prod_{l=1}^{2}
\bigg(\frac{\|w_{l}^{-1}\chi_{\Q}\|_{p_{l}^{\prime}(\cdot)}}{|\Q|}\bigg)^{p(x)}
\frac{u(x)}{(e+|x|)^{tnp_{-}}}\,dx \\
 & \quad \quad =R_{1}+ R_{2}.
\end{align*}

The estimate for $R_2$ is identical to the estimate for $L_2$.  To
estimate $R_1$, we again use \eqref{eqn:p-infty-Hest} for $\sigma_1$,
replacing $p_\infty$ with $s(\Q)$. 
Because $\sigma_2(\Q)<1$ we use a different estimate.   Since
$(p_2')_-(\Q)=(p_2)_+(\Q)^{'}$,  by Lemma~\ref{lemma:mod-norm},

\begin{equation*}
\bigg(\frac{\sigma_2(\Q)}{\|w_2^{-1}\chi_\Q\|_{p_2'(\cdot)}}\bigg)^{s(\Q)}
 \leq  \bigg(\sigma_2(\Q)^{1-\frac{1}{[(p_2)_+]'(\Q)}}  \bigg)^{s(\Q)}
 = \sigma_2(\Q)^{\frac{s(\Q)}{(p_2)_+(\Q)}}.  
\end{equation*}
We can now argue as in the estimate of $L_1$ to get
\begin{align*}
    R_{1}&\lesssim \sum_{(k,j)\in
           \mathscr{H}_{1b}}\int_{E_{j}^{k}}{\prod_{l=2,4}
\langle h_{l} \rangle_{\sigma_{\rho(l)},\Q}^{s(\Q)}
\sigma_{1}(\Q)^{\frac{s(\Q)}{(p_1)_{\infty}}}
\sigma_{2}(\Q)^{\frac{s(\Q)}{(p_{2})_+(\Q)}}}\\
    & \qquad \qquad \qquad \qquad  \times
      \prod_{J=1}^{2}\bigg(\frac{\|w_{J}^{-1}\chi_{\Q}
      \|_{p_{J}^{\prime}(\cdot)}}{|\Q|}\bigg)^{p(x)}u(x)\,dx \\
    & \leq \sum_{(k,j)\in \mathscr{H}_{1b}}
\prod_{l=2,4}\langle h_{l} \rangle_{\sigma_{\rho(l)},\Q}^{s(\Q)}
\sigma_{1}(\Q)^{\frac{s(\Q)}{(p_1)_{\infty}}}
\sigma_{2}(\Q)^{\frac{s(\Q)}{(p_{2})_+(\Q)}}\\
    &  \qquad \qquad \qquad \qquad
      \times\int_{\Q}{\prod_{J=1}^{2}\|w_{J}^{-1}\chi_{\Q} 
\|_{p_{J}^{\prime}(\cdot)}^{p(x)}|\Q|^{-2p(x)}u(x)\,dx}\\
    &\lesssim \sum_{(k,j)\in \mathscr{H}_{1b}}
\langle h_{2}\rangle_{\sigma_{1},\Q}^{(p_{1})_{\infty}}\sigma_{1}(\Q)
+\sum_{(k,j)\in \mathscr{H}_{1b}}\langle
      h_{4}\rangle_{\sigma_{2},\Q}^{(p_{2})_+(Q_j^k)}
\sigma_{2}(E_j^k).
\end{align*}     

The estimate for the first term in the last line is the same as the
estimate  for the $h_4$ term in  $J_{2}$.  Arguing as we did
for~\eqref{eqn:p-J3-est} and~\eqref{eqn:p2-J3-est},
\[  |(p_2)_+(Q_j^k) - p_\infty| \lesssim \frac{1}{\log(e+|x|)}.  \]
Then, since $\langle h_{4}\rangle_{\sigma_{1},\Q}\leq 1$, the estimate for
 the second term follows by
 ~\eqref{eqn:p2-J3-est},
 and  by Lemmas~\ref{lemma:p-infty-px}, \ref{lemma:wtd-max-bound} and \ref{lemma:infty-bound}:
    \begin{align*}
& \sum_{(k,j)\in \mathscr{H}_{1b}}\langle
      h_{4}\rangle_{\sigma_{2},\Q}^{(p_{2})_+(\Q)}
\sigma_{2}(E_j^k) \\
 & \qquad \qquad  \lesssim \sum_{(k,j)\in
   \mathscr{H}_{1b}}\int_{E_{j}^{k}}\langle
   h_{4}\rangle_{\sigma_{2},\Q}^{(p_{2})_\infty}\sigma_{2}(x)\,dx+
   \int_{\subRn}{\frac{\sigma_2(x)}{(e+\abs{x})^{tn(p_2)_-}}\,dx}+1\\
& \qquad \qquad  \lesssim \sum_{(k,j)\in
   \mathscr{H}_{1b}}\int_{E_{j}^{k}}\langle
   h_{4}\rangle_{\sigma_{2},\Q}^{(p_{2})_\infty}\sigma_{2}(x)\,dx+ 1.
    \end{align*}
Again, we estimate this last sum as in the final estimate for $J_2$.

\medskip

To estimate the sum over $\mathscr{H}_{2}$ we argue as we did before
for $J_3$,
dividing it into 
sums over $\mathscr{H}_{2a}$ and $\mathscr{H}_{2b}$.   
The estimate over $\mathscr{H}_{2a}$ is identical to  the estimate
over this set as before, replacing $h_1$ by $h_2$.  This yields
terms just like $L_1$ and $L_2$ above.  The estimate for the $L_2$
term is the same, as is the estimate for the $L_1$ term,  except
that in the final line the $h_2$ term is estimated like the $h_4$ term
since both $h_2,\,h_4 \leq 1$.  

\medskip

To estimate the sum over $\mathscr{H}_{2b}$, we can argue as before,
getting terms like $M_1$ and $M_2$, replacing $h_1$ by $h_2$.  The
estimate of the $M_2$ term is again the same.  To estimate the $M_1$
term we argue as before 
except we replace the exponent $r(\Q)$ by $p_\infty$.   But then the
final line of the estimate becomes
\[ \sum_{(k,j)\in \mathscr{H}_{2b}}
\langle h_{1}\rangle_{\sigma_{1},\Q}^{(p_{1})_\infty}\sigma_{1}(\Q)
+\sum_{(k,j)\in \mathscr{H}_{2b}}\langle
      h_{4}\rangle_{\sigma_{2},\Q}^{(p_{2})_{\infty}}
\sigma_{2}(\Q), \]
and both of these sums are estimated like the final estimate for
$J_2$.  This completes the estimate for $N_3$ and so of $I_4$.  This
completes the proof of Theorem~\ref{thm:main}.

\section{Proof of Theorem~\ref{thm:sio}}
\label{section:proof-sio}

Theorem~\ref{thm:sio} follows almost directly from
Theorem~\ref{thm:main}.  To prove it, we will need two estimates for
the Fefferman-Stein sharp maximal operator and an extrapolation
theorem in the scale of weighted variable Lebesgue spaces.  We first
recall the definition of the sharp maximal operator:  given
$f\in L^1_{loc}$, let
\[ M^\# f(x) = \sup_Q \avgint_Q |f(y)-\langle f \rangle_Q|\,dy\,
  \chi_Q(x), \]
where the supremum is taken over all cubes $Q$.   For $\delta>0$,
define $M^\#_\delta f(x) = M^\#(|f|^\delta)(x)^{\frac{1}{\delta}}$.
The first estimate relates the norm of $f$ and $M^\#$.  For a proof,
see Journ\'e~\cite{MR706075} or~\cite{cruz-uribe-martell-perez04}. 

\begin{proposition} \label{prop:wtd-sharp}
Given $w\in A_\infty$, $0<p<\infty$, and $0<\delta<1$, 
\[ \|f\|_{L^p_w} \lesssim \|M^\#_\delta f\|_{L^p_w}. \]
The implicit constant depends on $p$, $n$, $\delta$ and $w$. 
\end{proposition}

The second estimate is a pointwise inequality  proved in~\cite{MR2483720}.

\begin{proposition} \label{prop:sharp-est}
Given $0<\delta<1/2$  and a bilinear Calder\'on-Zygmund singular integral $T$, for all $f_1,\,f_2\in L^\infty_c$, 
\[ M^\#_\delta (T(f_1,f_2))(x) \lesssim \M(f_1,f_2)(x). \]
The implicit constant depends only on $T$, $\delta$ and $n$.
\end{proposition}

To apply these results we need to extend
Proposition~\ref{prop:wtd-sharp} to the scale of variable Lebesgue
spaces.  The following result was proved
in~\cite[Theorem~2.25]{CruzUribeSFO:2017km}.  The hypotheses are
somewhat technical, but they are the right generalization to prove
$A_\infty$ extrapolation~\cite{cruz-uribe-martell-perez04} in
this setting.   The result is stated in the abstract language of
extrapolation pairs; for more on this approach to Rubio de Francia
extrapolation, see~\cite{MR2797562}.  

\begin{proposition} \label{prop:Ainfty-extrapol}
Suppose for some $0<p<\infty$ and every $w_0\in A_\infty$, 
\begin{equation} \label{eqn:extrapol1}
\|f\|_{L^p_{w_0}} \lesssim \|g\|_{L^p_{w_0}} 
\end{equation}
for every pair of functions $(f,g)$ in a family $\F$ such that
$\|f\|_{L^p(w)}<\infty$.  Given $\pp \in \Pp_0$, suppose there
exists $s\leq p_-$ such that $w^s \in \A_{\pp/s}$ and the maximal
operator is bounded on $L^{(\pp/s)'}(w^{-s})$.  Then for $(f,g)\in \F$
such that $\|f\|_{L^\pp(w)}<\infty$,
\[ \|f\|_{L^\pp(w)} \lesssim \|g\|_{L^\pp(w)}. \]
\end{proposition}

\begin{proof}[Proof of Theorem~\ref{thm:sio}]
Fix $\vecpp$ as in the hypotheses and $\vec{w}\in \A_\vecpp$.  Since
$(p_j)_->1$, $p_->1/2$.  Let $s=1/2$; then by
Proposition~\ref{AP-charachterization} $w^s \in \A_{\pp/s}$, so
$w^{-s} \in \A_{(\pp/s)'}$.  Since $\pp \in LH$, so is $(\pp/s)'$.
Thus, by the weighted bounds for the maximal operator on variable
Lebesgue spaces (see~\cite{dcu-f-nPreprint2010}), $M$ is bounded on
$L^{(\pp/s)'}(w^{-s})$.  Therefore, the main hypothesis of
Proposition~\ref{prop:Ainfty-extrapol} holds.  

Fix $0<\delta<1/2$ and define the family of extrapolation pairs
\[ \F = \big\{ \big( \min(|T(f_1,f_2)|,N)\chi_{B(0,N)},
  M_\delta^\#(T(f_1,f_2))\big) : f_1,\, f_2 \in L^\infty_c, N>1
  \big\}. \]
Since 
\[ \min(|T(f_1,f_2)|,N)\chi_{B(0,N)} \in L^\infty_c \subset
L^p_{w_0} \]
 for any $p>0$ and $w_0\in A_\infty$, it follows
from Proposition~\ref{prop:wtd-sharp} that~\eqref{eqn:extrapol1} holds
for every pair in $\F$.   Similarly, we have
\[ \min(|T(f_1,f_2)|,N)\chi_{B(0,N)} \in L^\pp(w), \]
and so by
Propositions~\ref{prop:Ainfty-extrapol} and~\ref{prop:sharp-est},
\[ \|\min(|T(f_1,f_2)|,N)\chi_{B(0,N)} \|_{L^\pp(w)}
\lesssim \| M_\delta^\#(T(f_1,f_2))\|_{L^\pp(w)} 
\lesssim \|\M(f_1,f_2)\|_{L^\pp(w)}.  \]
If we take the limit as $N\rightarrow \infty$, then by Fatou's lemma (Lemma~\ref{lemma:fatou})
and Theorem~\ref{thm:main},
\[ \|T(f_1,f_2) \|_{L^\pp(w)}
\lesssim \|f_1\|_{L^\pap(w_1)} \|f_2 \|_{L^\pbp(w_2)}. \]
The desired conclusion now follows by a standard approximation
argument since $L^\infty_c$ is dense in 
$L^{p_j(\cdot)}(w_j)$,
$j=1,\,2$~\cite[Lemma~3.1]{CruzUribeSFO:2017km}. 
\end{proof}

\bibliographystyle{plain}

\bibliography{bilinear-maxvar}

\end{document}